\documentclass[leqno,12pt]{article} 
\usepackage[dvipdfmx,colorlinks=true]{hyperref}
\hypersetup{urlcolor=blue, citecolor=red}
\usepackage{amsmath, amssymb, bm}
\usepackage{amsthm} 
\usepackage{amsmath,amssymb,mathtools}
\usepackage{mathrsfs} 
\usepackage{ascmac}
\usepackage{graphicx,color}
\theoremstyle{definition} 
\newtheorem{Th}{\bf Theorem}[section] 
\newtheorem{Le}{\bf Lemma}

\newtheorem{Rem}{\bf Remark}

\newcommand{\N}{\mathbb{N} }

\newcommand{\R}{\mathbb{R} }

\newcommand{\LC}{\left ( }
\newcommand{\RC}{\right ) }
\newcommand{\LD}{\left \{ }
\newcommand{\RD}{\right \} }
\newcommand{\LZ}{\left | }
\newcommand{\RZ}{\right | }

\newcommand{\LN}{\left \|  }
\newcommand{\RN}{\right \| }

\DeclareMathOperator{\diag}{diag}
\DeclareMathOperator{\id}{id}

\DeclareMathOperator*{\argmax}{argmax}

\makeatletter
\def\address#1#2{\begingroup
\noindent\parbox[t]{7.8cm}{%
\small{\scshape\ignorespaces#1}\par\vskip1ex
\noindent\small{\itshape E-mail address}%
\/: #2\par\vskip4ex}\hfill%
\endgroup}%
\makeatother
\title{Cahn--Hilliard equation associated with hypergraph Laplacian} %
\author{Takeshi Fukao, Masahiro Ikeda, Shun Uchida}
\date{} %
\begin{document}

\maketitle


\begin{abstract}
A hypergraph Laplacian was introduced to investigate the structure of networks written as hypergraphs. 
It was shown that the behavior of solutions to an  evolution equation associated with the hypergraph Laplacian 
quite resembles that of solutions to the classical heat equation. 
Hence by replacing the Laplacian in PDEs with the hypergraph Laplacian,
we might be able to introduce various diffusion models on hypergraphs (discrete domains), 
whose behavior of solution is similar to PDEs. 
In this paper, we consider a system of  equations obtained by replacing the Laplacian 
in the original Cahn--Hilliard equation with the hypergraph Laplacian. 
Due to the nonlinearity and multivaluedness of the hypergraph Laplacian,
it is difficult to apply methods for the original Cahn--Hilliard equation 
to our problem.  
To cope with these difficulty, 
we shall introduce a new proof by using properties of the hypergraph Laplacian in this paper. 

\end{abstract}


\section{Introduction}

Multiple computers connected in a network, 
or multiple drones controlled while communicating with one another, 
can be described as graphs or hypergraphs. 
Describing the flow of physical quantities on such graphs is expected to lead to a deeper understanding of these networks. 
As an illustrative example, 
let us abstract everything and consider the diffusion of a virtual quantity analogous to heat. 
This can be modeled as a heat equation on a graph or hypergraph: 
the virtual heat diffuses over the graph or hypergraph as time progresses. 
One expects that diffusion proceeds more readily across strongly connected portions 
and more slowly across weakly connected ones.
Let us now consider more general parabolic equations. 
For instance, the Allen--Cahn equation and the Cahn--Hilliard equation, 
which describe phase transition and phase separation phenomena, 
are known to evolve in the direction of decreasing a double-well potential. 
Just as the heat equation on a graph or hypergraph involved a virtual physical quantity, 
the order parameter appearing in these two equations may likewise be regarded as a virtual quantity in the present context. 
Nevertheless, the Allen--Cahn and Cahn--Hilliard equations are recognized not merely as models of phase transition and phase separation but, more broadly, as mathematical models describing a wide variety of phenomena --- in particular, as continuous approximation models for phenomena involving free boundaries. 
Consequently, the study of these equations on graphs and hypergraphs is expected to yield diverse applications in the future. 
Here, we focus our attention on the Cahn--Hilliard equation.

The (viscosity) Cahn--Hilliard equation is the following system of partial differential equations: 
\begin{equation}
\begin{cases}
~~ \partial _t u - \Delta \mu = 0,  \\ 
~~ \mu = \tau \partial _t u - \Delta u + W ' (u) . 
\end{cases}
\label{C-H}
\end{equation}
This system was introduced by Cahn--Hilliard \cite{CH}
to describe two-phase separation process of a binary alloy.
Here $u $ is the order parameter, $\mu $ is the chemical potential, 
and $\tau $ is a viscosity parameter.
Moreover, $ W ' (u )$ is the derivative of a double-well potential $W(u)$.
A typical example of potentials is   
$W  (u) = \frac{1}{4} u ^4  - \frac{1}{2} u ^2  $ and $W ' (u) = u ^3 -  u  $. 
In this equation, we can observe the behavior of 
the order parameter $u$ which 
is separated into the pure states $u = 1$ and $u = -1$ (the minimizer of the double-well potential).
Recently, 
the cases of more general potentials $W$ have been investigated in many papers
(see, e.g., \cite{CMZ,  CF01, CF02, CGS, GMS, M} and references therein).
Here, the unknown functions $u$ and $\mu$ may be interpreted as scalar-valued functions on a given domain. 
However, in the present paper we extend these to vector-valued functions, and the reader should bear this in mind.

On the other hand, a hypergraph Laplacian $\mathcal{L} _{G,p} $ is introduced in Yoshida \cite{Yoshida00}
which comprises
a subdifferential of the Lovasz extension of submodular transformations
and written as a nonlinear set-valued operator on an Euclidean space $\R ^N $
(precise definitions of hypergraphs and the hypergraph Laplacian will be stated in  Section 2).
The hypergraph Laplacian is used to investigate the structure of networks written as hypergraphs,
for example, hypergraph clustering based on Cheeger's inequality (see \cite{IMTY, Yoshida00}) 
and PageRank (see \cite{TMIY})
and Ricci curvature of hypergraphs (see \cite{Aka, IKTU}).
In \cite{IU}, we consider an evolution equation 
associated with the hypergraph Laplacian $u' + \mathcal{L}_{G ,p } (u )  \ni 0 $
and we show that the behavior of solutions to $u' + \mathcal{L}_{G ,p }  (u )  \ni 0 $
quite resembles that of the solution to the classical heat equation with $p$-Laplacian 
$\partial _t u - \Delta _p u = 0$. 
Hence by replacing the Laplacian $\Delta $ in parabolic type PDEs with the hypergraph Laplacian $\mathcal{L}_{G,p }$,
we might be able to introduce various diffusion models on hypergraphs (discrete domains), 
whose behavior of solution is similar to parabolic type PDEs.

In this paper, we consider a system of  equations obtained by replacing the Laplacian 
$\Delta $ in the  Cahn--Hilliard equation with the hypergraph Laplacian $ \mathcal{L} _{G ,p }$.
\begin{equation}
\begin{cases}
~~ u' (t) + \mathcal{L} _{G,p} (\mu (t) ) \ni 0 
			~~&~~t \in (0,T) , \\
~~\mu (t) \in \tau u'(t) + \mathcal{L} _{G,p} (u(t)) + a \| u(t) \| ^{q-2} u(t)
			+ \pi (u(t)) - f(t) ~~&~~ t \in (0,T) , \\
~~u(0) = u_0 ,  
\end{cases}
\label{Eq}
\tag{P}
\end{equation}
where $ \tau \in [0 , \infty ) $, $a > 0 $, $q > 2 $ are given parameters 
and 
$\mathcal{L}  _{G, p } :\R ^N \to 2 ^ { \R ^N}  $ is the hypergraph Laplacian. 
Moreover, 
$ \pi : \R ^N \to \R ^N $ is a Lipschitz continuous function with 
a primitive function $k : \R ^N \to \R $ and 
$\| \cdot  \|$ is the canonical norm of $\R ^N $. 
Since the hypergraph Laplacian $\mathcal{L}  _{G, p } :\R ^N \to 2 ^ { \R ^N}  $ 
is a (nonlinear set-valued) operator on a Euclidean space $\R ^N $, 
\eqref{Eq} is a nonlinear set-valued system of ODEs 
and unknown functions are $u , \mu  : [0,T] \to \R ^N $
and a given function is $f : [0,T] \to \R ^N $,
where $T > 0 $ is a given length of the time interval.
This system possesses a double-well potential $W (u) = \frac{a}{q} \| u \| ^q - k (u)$. 
This approach follows the well-established framework of  multi-component Allen--Cahn and Cahn--Hilliard equations (see, e.g., \cite{RSK} and \cite{EL, E}).

 By analogy with the results of the Cahn--Hilliard equation \eqref{C-H}, this system \eqref{Eq} can be regarded as a preliminary model for the two-phase separation process on hypergraphs (discrete domains). Here, the double-well potential is simply adopted in the form used for the multi-component case, without adaptation to the hypergraph setting; accordingly, the equation should be understood as a first step toward describing such clustering behavior, rather than a complete description of the phenomenon. Naturally, the potential need not be restricted to this particular form and may instead be defined more generally as a nonlinear term, depending on the application. In this paper, however, we adopt the potential described above. 
Moreover, we expect that the order parameter varies smoothly across the hypergraph, in the sense that vertices close to a fixed vertex tend to take similar values, while vertices far from it tend to take increasingly different values. This behavior, if confirmed, would provide a further indication that the equation qualitatively reflects the clustering tendency of hypergraph networks, even though the potential considered here does not yet enforce a sharp two-phase separation at each vertex.

One of the methods for solving the Cahn--Hilliard equation
is to rewrite \eqref{C-H} into a single equation 
\begin{equation*}
\partial _t u - \tau \partial _t \Delta u + \Delta ^2 u = \Delta W ' (u )  
\end{equation*}
and use techniques for the parabolic equation with bilaplacian operator $\Delta ^2 $ (see, e.g., \cite{M}). 
However, this method does not seem to be valid for \eqref{Eq}
since the hypergraph Laplacian $\mathcal{L} _{G ,p } $ is not a linear operator 
and the square of a nonlinear operator $A^2 $ is not maximal monotone in general 
even if $A $ is maximal monotone.
Another approach is to consider 
\begin{equation*}
(-\Delta ) ^{-1 } \partial _t u  + \tau \partial _t  u -  \Delta  u = - W ' (u )  
\end{equation*}
and use an  abstract result \cite{CV} for doubly nonlinear evolution equation
(see, e.g., \cite{CF01, CF02}).
This method  also seems to be  unsuitable for \eqref{Eq}
since the growth condition for the inverse operator $( \mathcal{L} _{G ,p } )^{-1}$ is not clear
and therefore it is unknown whether the assumptions of \cite{CV} are satisfied.

To cope with these difficulty, 
we shall introduce a new proof by using properties of the hypergraph Laplacian in this paper. 
Roughly speaking, we first show the solvability of the following system, 
where the hypergraph Laplacian $\mathcal{L} _{G ,p }$ in \eqref{Eq} is replaced with 
its Yosida approximation $\partial \varphi _{\lambda }$:
\begin{equation*}
\begin{cases}
~~ u' (t) + \partial \varphi _{\lambda } (\mu (t) ) = 0  , \\
~~\mu (t) = \tau u' (t) + \partial \varphi _{\lambda } (u(t)) + a \| u(t) \| ^{q-2} u(t)
			+ \pi (u(t)) - f(t).
\end{cases}
\end{equation*}
Then by letting $\lambda \to 0 $, we assure the existence of solution to \eqref{Eq}. 
Note that in this method, we treat the term $\tau u' $ as a perturbation term
and then the case where $\tau > 0 $ is more difficult than the case where $\tau = 0 $,
unlike the usual situations.
In the next section, we state the definition of the hypergraph Laplacian 
and its basic properties. 
We first treat the case where $\tau = 0 $ in Section 3 
and  next consider the case where $\tau > 0 $ in Section 4.
In the final section, 
we discuss the convergence of the systems and the solutions 
as $\tau \to 0$.

\section{Preliminary}

Let $G = (V , E  ,w )$ be a (weighted) hypergraph, where
\begin{itemize}
\item $V = \{ 1, 2, \ldots, N \}$ is a finite set (set of vertices), 

\item $E $ is a family of subsets of $V $ with more than one element, 
i.e., $\# e \geq 2 $ for every $e \in E $, where $\# e $ 
is the number of elements contained in $e$ (set of hyperedges), 

\item $w : E \to (0, \infty )$ (weight of hyperedge).

\end{itemize}
This triplet $G$ represents a network in which the vertices labeled
from $1$ to $ N$ are connected by each  hyperedge $e\in E $. 
When  $e = \{ i,j \}$, that is, $e$ consists of just two elements,
then $e $ corresponds to a line segment connecting $i$-th and $j$-th vertices. 
Hence the usual graph can be represented by $ G$ with  $E$ satisfying
$\# e = 2$ for every  $e \in E$. 
The hypergraph is a generalization of the usual graph. When
$\# e \geq 3 $, then $e$ can be regarded as a group of multiple vertices.
For instance, the model of relationship of co-authorship between researchers and
communities in the social media can be described by the hypergraph.
The weight $w (e ) $ describes the strength of the connection of hyperedge $e\in E$
 (see Fig. \ref{Hypergraph}).
Throughout this paper, we assume that the hypergraph $G = (V , E  ,w )$ is connected.
That is, suppose that 
for every $ i , j \in V $, there exist some $ i_1 , \ldots , i _{n-1} \in V $ and $ e_ 1 , \ldots, e_n \in E$ such that 
$ i_{l-1} , i _{l } \in e _l $ holds for
any $ l = 1 , \ldots, n $, where $i_0 = i $ and $ i_n = j $.

\begin{figure}[h]
 \centering
 \includegraphics[keepaspectratio, scale=0.45]
      {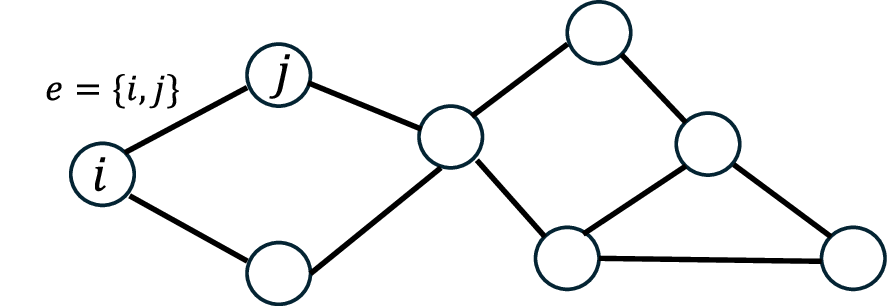}
      \hspace{4mm} 
 \includegraphics[keepaspectratio, scale=0.45]
      {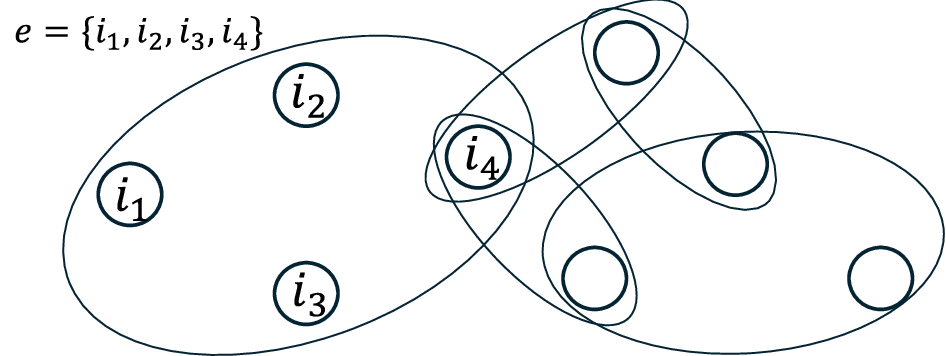}
 \caption[Usual Graph and Hypergraph]{When {\#}$e =2 $ for every $e\in E$,
each edge $e \in E $ can be regarded as a line segment connecting two vertices (left figure, called the {\it usual graph}).
Hypergraph is a generalization of usual graphs 
which represents the grouping of multiple members (right figure).
}
 \label{Hypergraph}
\end{figure}

As for the usual graph, we can define a graph Laplacian matrix by the following way.
Let 
\begin{equation}
w_{ij} := \begin{cases}
~~
w( \{ i , j \} ) ~~& ~~ \text{ if } \{ i, j \} \in E ,  \\
~~~~~0 ~~& ~~ \text{ if } \{ i, j \} \not \in E , 
\end{cases}
~~~~
d_ i := 
\sum_{j=1}^{N} w _{ ij } 
\label{usu-wei}
\end{equation}
and define square matrices of order $N$
by $W := (w _{ij} ) _{ 1 \leq i , j \leq N }$ and $D := \diag ( d_1 ,\ldots , d_ N  )$,
which are called a (weighted) adjacency matrix
and a (weighted) degree matrix, respectively. 
When we consider the random walk on a usual graph, 
we obtain the matrix $D ^{-1} (D - W )$ as the transition matrix. 
Here $\mathcal{L} := D -W $, which essentially describes the movement of the particles,
is called the graph Laplacian.
The graph Laplacian matrix has been used to investigate the structure of network of graphs 
(see, e.g., \cite{Chung, CDS})
and applied to the PageRank \cite{BP, Chung02}
and  the analysis of graph clustering \cite{Chung}.

On the other hand, a hypergraph Laplacian is defined as follows 
(see  Yoshida \cite{Yoshida00}, and  also \cite{FIU01, FIU02, IU}). 
For each hyperedge $e \in E $, we define $ f_ e : \R ^N \to [0 , \infty )  $ by
\begin{equation*}
f_e (x) := \max _{ i , j \in e } |x _i - x _ j |
\end{equation*}
and 
\begin{equation*}
\varphi _{G ,p } (x ) := \frac{1}{p} \sum_{e\in E }  w (e) ( f _e (x ) ) ^{p} ~~~~~p \in [1, \infty ) .  
\end{equation*}
When $G$ is a usual graph, we can write
\begin{equation*}
f_e (x) :=  |x _i - x _ j |~~~\text{where } e = \{ i, j \}
\end{equation*}
and 
\begin{equation*}
\varphi _{G ,p } (x ) := \frac{1}{2p} \sum_{ i , j =1  }^N  w _{ij} |x _ i - x_ j | ^{p} ,
\end{equation*}
where $w_ {ij }$ is defined in \eqref{usu-wei}.
Then it is easy to see that
$\varphi _{G, p }  : \R ^N \to \R $ with $p >1$ is Fr\'{e}chet differentiable 
and especially the differential of $ \varphi _{G, 2 } $
can be represented by the usual graph Laplacian $\mathcal{L} = D - W   $.
On the other hand,
when $ G$ is an ``essential'' hypergraph  
(i.e., there is at least one $e \in E $ with  $\# e \geq 3 $),  
$f_e $ and $\varphi _{G,p } $ are not differentiable in usual sense. 
However, since these are continuous and convex functions, 
we can define subdifferential of them. 
Here let $\phi : \R ^N \to ( - \infty , + \infty ] $ be a 
proper ($\phi \not \equiv + \infty $) lower semi-continuous convex function. 
Then its subgradient  at $x \in \R ^N $ is defined by 
\begin{equation*}
\partial \phi (x) := \{ \eta \in \R ^N ; ~ \eta \cdot (z - x ) \leq \phi (z) - \phi (x) ~~\forall z \in \R^N \} . 
\end{equation*}
Compared with the usual derivative, 
the subgradient $\partial \phi (x )$ possibly contains several elements in general. 
For example, the subgradient of  $\phi (x ) = \| x \|$ at $x = 0 $ 
coincides with 
\begin{equation*}
\partial \phi (0 ) := \{ \eta \in \R ^N ; ~ \| \eta \| \leq 1 \} . 
\end{equation*}
Then we call a set-valued mapping from $x $ to $\partial \phi (x)$ 
the subdifferential operator $\partial \phi : \R ^ N \to 2 ^{\R ^N }$. 
Since $\varphi _{G,p } $ is a continuous and convex function, 
we can define subdifferential of $\varphi _{G,p } $.
We here define the hypergraph Laplacian by  $ \mathcal{L}_{G ,p } = \partial \varphi _{G,p }$, 
i.e., the subdifferential operator of $\varphi _{G ,p }$.  
Since the hypergraph Laplacian is a genuine generalization  of the graph Laplacian matrix $\mathcal{L}$, 
our problem \eqref{Eq} includes a problem on a usual graph,
not only on the hypergraph.

By virtue of  the  maximum rule of subdifferential
(see, e.g., Proposition 2.54  \cite{M-N})
and the chain rule of  the subdifferential (see \cite{CLT}), 
we can obtain the explicit formula of $\partial f _ e $ and 
the hypergraph Laplacian $ \mathcal{L}_{G ,p } = \partial \varphi  _{G ,p  }$. 
Let $B _e \subset \R ^N $ be defined by 
\begin{align*}
B_e & := \text{conv}  \{ \bm{1} _{i }- \bm{1} _ {j} ;~ i , j  \in e  \}   \\
&=  \text{conv}  \LD (\ldots , 0 , \stackrel{i}{\stackrel{\vee}{1}},0, \ldots  ,0 , \stackrel{j}{\stackrel{\vee}{- 1}},0,  \ldots  )
\in \R^ N  ; ~ i ,j   \in e \RD ,   
\end{align*}
where $\bm{1} _ {i } $ is the $i$-th unit vector of the canonical basis of $\R ^N $.
Then  
$\partial f_e (x)$ and 
$\partial \varphi _{G,p } (x) $ can be represented by 
\begin{align*}
\partial f_ e  (x)  &=
\argmax _{b \in B_e} b \cdot x  = \left\{ b_e  \in B_e ;~~b_e \cdot x = \max _{b \in B_e} b \cdot x \right\}  ,   \\
\mathcal{L}_{G,p } (x ) = \partial \varphi _{G, p} (x) 
&=  \sum_{e \in E } w(e) (f_e (x) ) ^{p-1} \partial   f_e (x) \notag \\
	&= \LD \sum_{e \in E } w(e) (f_e (x) ) ^{p-1} b_e ;~b_e \in \argmax _{b \in B_e } b \cdot x  \RD . 
\end{align*}

\begin{Rem}
Since $\mathcal{L} _{G,p } = \partial \varphi _{G,p }$ is an (multi-valued nonlinear) operator on $\R ^N $
(the dimension $N$ coincides with the number of vertices in $V$), 
\eqref{Eq} is an (set-valued nonlinear) ODE system, not a PDE
unlike the usual  Cahn--Hilliard equation. 
Namely, \eqref{Eq} is a system with respect to the solution $u , \mu : [0,T ] \to \R ^N $. 
\end{Rem}

In  \cite{IU}, we found basic properties of the hypergraph Laplacian.
We here state some of them for later use.
Let $\bm{1}  = (1, \ldots, 1 ) $. 
By the definition of $B_e $, we easily get $ b \cdot \bm{1} = 0 $ 
for any $b \in B_e $ and $e \in E $. 
Since $\partial f _e (x) \subset B_e $ for every $e \in E $ and $x \in \R ^N $, 
we have 
\begin{equation*}
 \eta \cdot \bm{1} = 0 
~~~~\forall x \in \R ^N
~~~~\forall \eta \in \partial f_e  (x)
\end{equation*}
and then 
\begin{equation}
 \eta \cdot \bm{1} = 0 
~~~~\forall x \in \R ^N  
~~~~\forall \eta \in \mathcal{L}_{G,p }  (x) .
\label{times1}
\end{equation}
Next, it is obvious that $\| b \| \leq 2 $ and $f _ e (x ) \leq 2 \| x \|$ hold 
for every $b \in B_e $, $e \in E $, and $x \in \R ^N $.
Then there exists a constant $\kappa > 0 $ such that 
\begin{equation}
\varphi _{G,p } (x) \leq \kappa \| x  \| ^p , 
~~~~~
\| \eta  \| \leq \kappa \| x \| ^{p-1} 
\label{bounded} 
\end{equation}
for every 
$ x \in \R ^N  $ and $ \eta \in \mathcal{L} _{G,p } (x)$. 
Moreover, by the explicit formula of $\partial f _e $ and $\partial \varphi _{G, p }$, 
we obtain 
\begin{equation*}
 \eta \cdot x = f_e (x) 
~~~~\forall x \in \R ^N
~~~~\forall \eta \in \partial f_e  (x)
\end{equation*}
and 
\begin{equation}
 \eta \cdot x = p \varphi _{G,p } (x ) 
~~~~\forall x \in \R ^N  
~~~~\forall \eta \in \mathcal{L}_{G,p }  (x) .
\label{timesx}
\end{equation}
We also see the following (see Theorem 2.6 of \cite{IU}):
\begin{Le}
Assume that $G = (V, E ,w )$ is connected. 
For $x = (x _1 , \ldots , x _ N )\in \R ^N $ and $\bm{1} = (1 ,\ldots, 1 )$, let 
\begin{equation*}
\bar{x} := \LC \frac{1}{N} \sum_{i=1}^{N} x _i   \RC \bm{1} . 
\end{equation*}
Then there exists some constant $\kappa ' > 0 $ such that 
\begin{equation}
\| x - \bar{x} \| ^p \leq \kappa ' \varphi _{G,p } (x )
\label{Poincare}
\end{equation}
 for every $ x \in \R ^N $. 
\end{Le}

The hypergraph Laplacian is defined as a subdifferential operator. 
Then we can apply the standard results for the maximal monotone operator to the hypergraph Laplacian. 
We here state some of them. 
For $\lambda >0 $, we define 
\begin{equation}
\varphi _{\lambda } (x) 
:= 
\inf _{y \in \R ^N }
\LC
\frac{1}{2 \lambda }  \| x - y \| ^2 + \varphi _{G ,p } (y ) \RC , 
\label{MY}
\end{equation}
which is called the Moreau--Yosida regularization of $\varphi _{G,p }$. 
Then $\varphi _{\lambda } : \R ^N \to [ 0 , \infty ) $ becomes a continuous convex function. 
Moreover,  $\varphi _{\lambda } : \R ^N \to [ 0 , \infty ) $ is Fr\'{e}chet differentiable on $\R ^N $
and its derivative (which coincides with its subdifferential) is equal to the Yosida approximation of $\partial \varphi _{G,p }$  
(see Theorem 2.9 of \cite{Bar}). 
Here, the resolvent of $\partial \varphi _{G,p } $ is defined by 
$J_{\lambda } x := ( \id + \lambda \partial \varphi _{G,p }) ^{-1}$ with $\lambda >0 $, 
i.e., $y = J _{\lambda } x $ is a unique solution to 
\begin{equation*}
 y + \lambda \partial \varphi _{G,p } (y) \ni x 
\end{equation*}
for given $x \in \R ^N$
and the Yosida approximation of $\partial \varphi _{G,p } $ is defined by 
\begin{equation*}
\partial \varphi _{\lambda } (x) := \frac{x - J _{\lambda }x }{ \lambda } .
\end{equation*}
It is easy to see that 
\begin{equation*}
 \partial \varphi _{\lambda } (x ) \in \partial \varphi _{G,p } (J _{\lambda } x )
\end{equation*}
holds for every $x \in \R ^N $ and $\lambda > 0 $.  
We can get the followings (see Proposition 2.3, Theorem 2.9, and Proposition 2.2 of \cite{Bar}, respectively):
\begin{Le}
The mapping $\partial \varphi _{\lambda } : \R ^N \to \R ^N $ is single-valued and Lipschitz continuous. 
More precisely,
\begin{equation}
\| \partial \varphi _{\lambda } (x) -  \partial \varphi _{\lambda } (y) \|
\leq \frac{1}{\lambda } \| x - y \| 
\label{Lipschitz} 
\end{equation}
holds for any $x , y \in \R ^N $. 
\end{Le}
\begin{Le}
For every $x\in \R ^N $, 
\begin{equation}
\varphi _{G,p} (J _{\lambda } x ) \leq \varphi _{\lambda } (x ) \leq \varphi _{G,p  } (x ). 
\label{MY-01}
\end{equation}
\end{Le}
\begin{Le}
For every $x\in \R ^N $,  
\begin{equation}
\| \partial \varphi _{\lambda } (x ) \|  \leq \|  ( \partial \varphi _{G,p }) ^{\circ } (x)  \|
\label{MY-02}
\end{equation}
holds,  where $  ( \partial \varphi _{G,p }) ^{\circ } $ is the 
minimal section of $\partial \varphi _{G,p } $, namely, 
$  ( \partial \varphi _{G,p }) ^{\circ }   $ satisfies $( \partial \varphi _{G,p }) ^{\circ } (x)  \in \partial \varphi _{G,p } (x)$ and
\begin{equation*}
\|  ( \partial \varphi _{G,p }) ^{\circ } (x)  \|  \leq \| \eta \| ~~~~\forall \eta \in  \partial \varphi _{G,p } (x). 
\end{equation*}

\end{Le}

Since the subdifferential operator becomes maximal monotone, 
we can use the following facts:

\begin{Le}
\label{monotonicity}
Let $\phi = \varphi _{G,p }$ or $\phi = \varphi _{\lambda }$ with $\lambda > 0 $. 
Then 
\begin{equation*}
 ( \eta _1 - \eta _2 ) \cdot (x_1 - x_2 ) \geq 0 
\end{equation*}
holds for any $ x_i \in \R ^N $ and $\eta _i \in \partial \phi (x_ i )$ with $i=1,2 $
(see Theorem 2.8 of \cite{Bar} and the definition of the monotonicity). 
\end{Le}

\begin{Le}
\label{demiclosedness}
Let $\phi = \varphi _{G,p }$ or $\phi = \varphi _{\lambda }$ with $\lambda > 0 $,
then the operator $\partial \phi $ is demiclosed in $L^2 (0,T ;\R ^N )$.  
Namely, assume $ x_n  \in L^2 (0, T ; \R ^N )  $,  $\eta _ n  \in L ^2 (0,T ; \R ^N )$,
and $ \eta _ n  (t ) \in \partial \phi  (x _n  (t))$ for a.e. $t \in [0,T]$ with $n \in \N $.  
Let $x _n \to x $ strongly in $L ^2 (0,T ; \R ^N ) $ and   
$\eta _n  \rightharpoonup \eta $ weakly in $L ^2 (0,T ; \R ^N ) $. 
Then $\eta (t) \in \partial \phi (x (t ))$ holds for a.e. $t\in [0,T ]$ (see Proposition 3.4 of \cite{BCP}). 
\end{Le}

\begin{Le}
\label{wwlimsup}
Let $\phi = \varphi _{G,p }$ or $\phi = \varphi _{\lambda }$ with $\lambda > 0 $. 
Assume $ x_n  \in L^2 (0, T ; \R ^N )  $,  $\eta _ n  \in L ^2 (0,T ; \R ^N )$,
and
$ \eta _ n  (t ) \in \partial \phi  (x _n  (t))$ for a.e. $t \in [0,T]$ and any $n \in \N $.  
Let $x _n \rightharpoonup x $ and $\eta _n  \rightharpoonup \eta $
weakly in $L ^2 (0,T ; \R ^N ) $ and 
\begin{equation*}
\limsup _{n\to \infty} \int_{0}^{T}  x _n (t)  \cdot  \eta _n (t) dt \leq 
\int_{0}^{T} x (t) \cdot \eta (t)  dt  . 
\end{equation*}
Then $\eta (t) \in \partial \phi (x (t ))$ holds for a.e. $t\in [0,T ]$ (see Lemma 1.2 of \cite{BCP}). 
\end{Le}

\begin{Le}
Let $\phi = \varphi _{G,p }$ or $\phi = \varphi _{\lambda }$ with $\lambda > 0 $. 
Assume $ x \in W ^{1,2 } (0, T ; \R ^N )  $,  $\eta \in L ^2 (0,T ; \R ^N )$,
and $ \eta (t ) \in \partial \phi (x (t))$ for a.e. $t \in [0,T]$. 
Then $\phi (x) : [0,T] \to \R $ is absolutely continuous and 
\begin{equation}
\eta (t) \cdot x'(t) = \frac{d}{dt} \phi (x (t )) 
\label{Chain}
\end{equation}
holds for a.e. $t \in [0,T ]$ (see Lemma 4.1 of \cite{Bar}). 
\end{Le}
We also review Gronwall's inequality (see Lemme A.5 of \cite{Bre}).
\begin{Le}
Let $\Xi : [0,T ] \to \R $ be continuous and satisfy 
\begin{equation*}
\frac{1}{2} |\Xi (t)| ^2 \leq \frac{1}{2} |a| ^2 + \int_{0}^{t} g (s ) \Xi (s) ds ~~~\forall t \in [0,T ]   
\end{equation*}
with some constant $a $ and non-negative function $g \in L^1 (0,T ; \R )$. 
Then $\Xi $ satisfies 
\begin{equation*}
|\Xi (t)| \leq |a| + \int_{0}^{t} g (s ) ds ~~~\forall t \in [0,T].  
\end{equation*}
\label{Gronwall}
\end{Le}

\section{Case of $\tau = 0 $}

In this section, we consider the case of $\tau = 0 $, i.e., 
\begin{equation*}
\begin{cases}
~~ u' (t) + \mathcal{L} _{G,p} (\mu (t) ) \ni 0 
			~~&~~t \in (0,T) , \\
~~\mu (t) \in  \mathcal{L} _{G,p} (u(t)) + a \| u(t) \| ^{q-2} u(t)
			+ \pi (u(t)) - f(t) ~~&~~ t \in (0,T) , \\
~~u(0) = u_0 ,  
\end{cases}
\end{equation*}
and prove the following fact: 
\begin{Th}
\label{TH01}
For every  $u _ 0 \in \R^ N $ and  $f \in W ^{1,1} (0,T ;\R ^N )$, 
\eqref{Eq} with $\tau = 0 $ possesses at least one solution 
satisfying
\begin{align*}
u \in W ^{1, \infty } (0,T ;\R ^N) , 
~~~\mu \in L ^{\infty} (0,T ;\R ^N).
\end{align*}
\end{Th}

\begin{proof}
Let $\partial \varphi _{\lambda }$ be the Yosida approximation of
$\partial \varphi _{G, p } $ with a parameter $\lambda > 0 $
and consider the following equations:
\begin{equation}
\begin{cases}
~~ u' (t) + \partial \varphi _{\lambda } (\mu (t) ) = 0 
			~~&~~t \in (0,T) , \\
~~\mu (t) = \partial \varphi _{\lambda } (u(t)) + a \| u(t) \| ^{q-2} u(t)
			+ \pi (u(t)) - f(t) ~~&~~ t \in (0,T) , \\
~~u(0) = u_0 .
\end{cases}
\label{EqLam}
\tag{P$_ \lambda $}
\end{equation}
We define a mapping $\mathscr{S} : L ^{\infty } (0, T' ; \R ^ N )
\to  L ^{\infty } (0, T' ; \R ^ N )$ in the following procedure
($T' \in ( 0 , T ] $ will be fixed later): 
\begin{itemize}
\item[i)] Fix $ \mu \in  L ^{\infty } (0, T' ; \R ^ N )$ and define 
$u _{\mu } $ by
\begin{equation*}
\begin{cases*}
~~ u' _{\mu } (t) = -  \partial \varphi _{\lambda } (\mu (t) ), \\
~~ u_{\mu}(0) = u_0. 
\end{cases*}
\end{equation*} 
By \eqref{bounded} and \eqref{MY-02}, we note that  $u _{\mu }$ belongs to $W ^{1, \infty } (0, T' ; \R ^ N )$. 

\item[ii)] By using $u _ {\mu } $ in i), we 
define $\mu _u \in L ^{\infty } (0, T' ; \R ^ N )$
by 
\begin{equation*}
\mu _ u  (t) = \partial \varphi _{\lambda } (u _{\mu }(t))
		+ a \| u_{\mu }(t) \| ^{q-2} u_{\mu }(t)
			+ \pi (u_{\mu } (t)) - f(t) .
\end{equation*}
To define $\mu _u \in L ^{\infty } (0, T' ; \R ^ N )$, 
we here assume $ f \in L ^{\infty } (0, T' ; \R ^ N ) $. 

\item[iii)] Define 
$\mathscr{S} : L ^{\infty } (0, T' ; \R ^ N )
\to  L ^{\infty } (0, T' ; \R ^ N )$ by 
$\mathscr{S} (\mu) :=  \mu _ u $.
\end{itemize}
If we can show that $\mathscr{S}$ is a contraction map in $L ^{\infty } (0, T' ; \R ^ N )$
with some sufficiently small $T  ' > 0 $, the fixed point 
$ \mu \in  L ^{\infty } (0, T' ; \R ^ N )$ and $ u _{\mu } \in W ^{1, \infty } (0, T ' ; \R ^N )$
becomes a solution to \eqref{EqLam}.

Define 
\begin{equation}
K_M := \LD \mu \in L ^\infty (0,T ' ; \R ^N ); ~~\sup _{0\leq t \leq T' } \| \mu (t) \| 
	\leq M \RD
	\label{KM}
\end{equation}
with $M > 0 $. 
By the first equation of \eqref{EqLam} and \eqref{bounded}, \eqref{MY-02}, 
we get 
\begin{align*}
\sup _{0 \leq t \leq T' } \| u' _{\mu}(t) \|
&\leq \sup _{0 \leq t \leq T' } \| \partial \varphi _{\lambda  } (\mu (t)))\|
\leq \sup _{0 \leq t \leq T' } \|  (\partial \varphi _{G,p })  ^{\circ }  (\mu (t)) \|
\leq \kappa \sup _{0 \leq t \leq T' } \| \mu (t) \| ^{p-1} \\
&\leq \kappa M ^{p-1}.
\end{align*}
Hence we have 
\begin{align*}
\| u_{\mu} (t) - u_0 \| \leq \int_{0}^{t} \| u '_{\mu} (\sigma )\| d\sigma 
\leq T' \kappa M ^{p-1} 
\end{align*}
and 
\begin{equation}
\sup _{0 \leq t \leq T' }\| u_{\mu} (t) \| 
\leq \| u_0 \| + T' \kappa M ^{p-1} .
\label{u_mubound} 
\end{equation}
From the second equation of  \eqref{EqLam},  
\begin{align*}
\| \mu _ u  (t) \| 
	& \leq \|  (\partial \varphi _{G,p } ) ^{\circ } (u _{\mu } (t )) \|
	+ a \| u _{\mu } (t) \| ^{q-1}
	+ \| \pi (u _{\mu } (t)) - \pi (0 )\| + \| \pi (0) \|
	+ \| f(t) \| \\
	&\leq \kappa \| u _{\mu } (t ) \| ^{p-1} 
	+ a \| u _{\mu } (t) \| ^{q-1}
	+ L \| u _{\mu } (t) \| + \| \pi (0) \|
	+ \| f(t) \| , 
\end{align*}
where $L > 0 $ is a Lipschitz constant of $\pi : \R ^N \to \R ^N $. 
Then we obtain 
\begin{align*}
\sup _{0 \leq t \leq T'}\| \mu _ u  (t) \| 
	&\leq \kappa \LC \| u_0 \| + T' \kappa M ^{p-1} \RC  ^{p-1} 
	+ a \LC \| u_0 \| + T' \kappa M ^{p-1} \RC ^{q-1} \\
	& \hspace{2cm} + L \LC \| u_0 \| + T' \kappa M ^{p-1} \RC + \| \pi (0) \|
	+\sup _{0 \leq t \leq T'} \| f(t) \| .
\end{align*}
By fixing  
\begin{align*}
M& := 
\kappa \LC \| u_0 \| + 1 \RC  ^{p-1} 
	+ a \LC \| u_0 \| + 1 \RC ^{q-1} \\
	& \hspace{2cm} + L \LC \| u_0 \| + 1 \RC + \| \pi (0) \|
	+\sup _{0 \leq t \leq T'} \| f(t) \|  , \\
T'&:= 1 / 	\kappa M ^{p-1} , 
\end{align*}
we obtain 
$ \sup _{0 \leq t \leq T'}\| \mu _ u  (t) \|  \leq M $, 
namely, $ \mu \in K_M \Rightarrow \mathscr{S} ( \mu ) \in K_M $ holds. 
Here by \eqref{u_mubound}, $u _{\mu } $ satisfies 
\begin{equation}
\sup _{0 \leq t \leq T' }\| u_{\mu} (t) \| 
\leq \| u_0 \| + 1
\label{u_mubound02} 
\end{equation}
for every $\mu \in K_M $.

Next we show that $\mathscr{S}$ is a contraction mapping. 
Let $\mu _1 , \mu _2 \in K_M $. 
Then by the equation, $u _{ \mu _1 } $ and $u _{\mu _2 }$ satisfy 
\begin{align*}
\| u'  _{ \mu _1 } (t) - u' _{\mu _2 } (t) \|
= \| \partial \varphi _{ \lambda } (\mu _1 (t) ) -\partial \varphi _{ \lambda } (\mu _2 (t)) \|
\leq \frac{1}{\lambda } \| \mu _1 (t) - \mu _ 2 (t) \|  , 
\end{align*}
where we use \eqref{Lipschitz}. 
This immediately leads to 
\begin{align}
\sup _{ 0\leq t \leq T' } \| u  _{ \mu _1 } (t) - u _{\mu _2 } (t) \|
\leq \frac{T' }{\lambda } \sup _{ 0\leq t \leq T' } \| \mu _1 (t) - \mu _ 2 (t) \|  . 
\label{est-difference}
\end{align}
From the second equation of \eqref{EqLam}, 
\begin{align*}
\| \mu _{u_1} (t) - \mu _{u_2} (t) \|
&\leq \frac{1}{\lambda } \| u_{\mu _1 } (t) - u_{\mu _2 } (t) \|
+ 
a \LN \| u_{\mu _1} (t) \| ^{q-2} u_{\mu _1}(t) -\| u_{\mu _2} (t) \| ^{q-2} u_{\mu _2} (t) \RN \\
&\hspace{3cm} +
L \| u_{\mu _1} (t) -u_{\mu _2} (t) \| . 
\end{align*}
By the Taylor expansion of $x \mapsto \| x\| ^{q-2} x $ with $q > 2 $, 
there exists some function $\theta : [0,T'] \to [ 0,1 ]$
such that 
\begin{align*}
&\LN \| u_{\mu _1} (t) \| ^{q-2} u_{\mu _1} (t) -\| u_{\mu _2}(t)  \| ^{q-2} u_{\mu _2} (t) \RN \\
&  = {} 
 (q-2 )\| \theta (t) u_{\mu _1} (t)+ (1- \theta (t))u_{\mu _2} (t) \| ^{q-2}  \|  u_{\mu _1} (t) - u_{\mu _2} (t) \|  \\
&  \leq {}  
(q-2 ) 
(  \|  u_{\mu _1} (t) \| + \| u_{\mu _2} (t) \| ) ^{q-2} 
\LN u_{\mu _1} (t) - u_{\mu _2} (t)  \RN \\
&  \leq {}   
 2^{q-2} (q-2 ) 
 ( \| u_0  \| +1 ) ^{q-2}   
\LN u_{\mu _1} (t) - u_{\mu _2} (t)  \RN ,
\end{align*}
where we apply \eqref{u_mubound02} to the above. 
Henceforth, let 
$\ell (\| u_0 \|) :=
 2^{q-2} (q-2 ) 
 ( \| u_0  \| +1 ) ^{q-2}   $. 
From  \eqref{est-difference},  we can derive 
\begin{align*}
& \sup _{0 \leq t \leq T' } \| \mu _{u_1} (t) - \mu _{u_2} (t) \| \\
& \leq  \frac{1}{\lambda } \sup _{0 \leq t \leq T' }  \| u_{\mu _1 } (t) - u_{\mu _2 } (t) \| + 
a\ell (\| u_0  \|) \sup _{0 \leq t \leq T' } \LN u_{\mu _1} (t) - u_{\mu _2} (t)  \RN \\
&\hspace{3cm} +
L \sup _{0 \leq t \leq T' } \| u_{\mu _1} (t) -u_{\mu _2} (t) \| \\
&
\leq \frac{T'}{\lambda ^2 } \sup _{0 \leq t \leq T' }  \| \mu _1  (t) - \mu _2  (t) \|
+ 
\frac{aT'\ell (\| u_0  \|)}{\lambda }  \sup _{0 \leq t \leq T' } \| \mu _1 (t)-\mu _2 (t)  \| \\ 
&\hspace{3cm} +
\frac{LT'}{\lambda } \sup _{0 \leq t \leq T' } \| \mu _1 (t) -\mu _2 (t) \| . 
\end{align*}
We here fix $T '> 0$ again by 
\begin{equation*}
T':= 
\min \LD 
\frac{1}{2} 
\LC
\frac{1}{\lambda ^2 }
+
\frac{a\ell (\| u_0  \|)}{\lambda }
+
\frac{L}{\lambda } \RC ^{-1} , 
\frac{1 }{	\kappa M ^{p-1}  } 
\RD ,  
\end{equation*}
then 
\begin{equation*}
\sup _{0 \leq t \leq T' } \| \mu _{u_1} (t) - \mu _{u_2} (t) \|
\leq 
\frac{1}{2}  \sup _{0 \leq t \leq T' } \| \mu _1 (t) -\mu _2 (t) \| . 
\end{equation*}
Therefore for a sufficiently small $T '> 0 $, 
$\mathscr{S} $ is a contraction map on $K_M $
and 
the existence of a time local solution to \eqref{EqLam} is assured.

To discuss the convergence as $\lambda \to 0 $, 
we next show the uniform boundedness of solutions
with respect to the parameter $\lambda > 0 $. 
Let $( u _{\lambda } , \mu _{\lambda })$ be a solution to \eqref{EqLam} with $\lambda > 0 $. 
Here, it is easy to see that $\varphi _{G ,p } (0 ) = 0 $ 
and $0 \in \partial \varphi _{G ,p } (0 )$ by the definition of the subdifferential.
Then we have $ J _{\lambda }  0  = 0 $ and $ \partial \varphi _{\lambda } (0) = 0 $. 
By Lemma \ref{monotonicity}, 
\begin{equation*}
\partial \varphi _{\lambda }  (\mu _{\lambda } (t)) \cdot \mu _{\lambda } (t)
=
( \partial \varphi _{\lambda }  (\mu _{\lambda } (t)) -  \partial \varphi _{\lambda }  (0)  ) \cdot  ( \mu _{\lambda } (t) -0 ) 
\geq 0 . 
\end{equation*}
Multiplying the first equation of \eqref{EqLam} by $\mu _{\lambda }(t) $, 
we get 
\begin{equation*}
u' _{\lambda } (t) \cdot \mu _{\lambda }  (t) \leq 0 .
\end{equation*}
Multiplying the second equation of \eqref{EqLam} by $u ' _{\lambda } (t)$, 
we obtain 
\begin{equation*}
u' _{\lambda }(t) \cdot \mu _{\lambda } (t)
=
\frac{d}{dt} \varphi _\lambda  (u _{\lambda }(t))
+ \frac{a}{q} \frac{d}{dt} \| u _{\lambda }(t) \| ^q 
+  \frac{d}{dt} k (u _{\lambda } (t)) - f(t) \cdot u '_\lambda (t) 
\end{equation*}
for a.e. $t \in (0, T' )$ (we here use \eqref{Chain}).
Hence 
\begin{equation*} 
\frac{d}{dt} \varphi _\lambda  (u _{\lambda }(t))
+ \frac{a}{q} \frac{d}{dt} \| u _{\lambda }(t) \| ^q 
+  \frac{d}{dt} k (u _{\lambda } (t)) \leq f(t) \cdot u '_\lambda (t) . 
\end{equation*}
Integrating this inequality between $[0, t]$, we have 
\begin{align*}
\varphi _\lambda  (u _{\lambda }(t))
+ \frac{a}{q}  \| u _{\lambda }(t) \| ^q 
+  k (u _{\lambda } (t)) 
 \leq 
\int_{0}^{t} f(s) \cdot u '_\lambda (s)ds 
+
\varphi _\lambda  (u _0 )
+ \frac{1}{q} \| u _0 \| ^q 
+   k (u _ 0 ) .
\end{align*}
Since we assume 
$f \in W ^{1,1} (0,T ; \R ^N )$, 
\begin{align*}
\int_{0}^{t} f(s) \cdot u '_\lambda (s)ds 
&=
f(t) \cdot u_ \lambda (t) 
-
f(0) \cdot u_ 0
- \int_{0}^{t} f'(s) \cdot u _\lambda (s)ds \\
&
 \leq 
C _1 \| f(t) \| ^{q' }+ \frac{a}{2q} \|  u_ \lambda (t) \| ^q
+
\| f(0) \| \| u_ 0 \|
+ \int_{0}^{t} \| f'(s) \| \| u _\lambda (s) \| ds 
\end{align*}
by the integration by parts, 
where $q ' := q / (q -1 ) $ is the H\"{o}lder conjugate of $q $. 
Here and henceforth, 
$c _1 , C _1  > 0 $ are general constants independent of $\lambda , T' $. 
From $\varphi _{\lambda } \geq 0$ and \eqref{MY-01}, 
we can derive 
\begin{align*}
& \frac{1}{2q}  \| u _{\lambda }(t) \| ^q 
+  k (u _{\lambda } (t)) \\
&
 \leq 
C_1  \| f(t) \| ^{q' }
+
\| f(0) \| \| u_ 0 \|
+ \int_{0}^{t} \| f'(s) \| \| u _\lambda (s) \| ds
+
\varphi   (u _0 )
+ \frac{1}{q} \| u _0 \| ^q 
+   k (u _ 0 ) .
\end{align*}
By the Lipschitz continuity of $\pi $, 
\begin{align*}
|k (x) | 
&\leq |k (x) - k (0 )| + |k(0)| \\
&\leq \LZ  \int_{0}^{1}  \frac{d}{ds} k (s x) ds \RZ + |k(0)|
	= \LZ  \int_{0}^{1}  \pi (s x) \cdot x ds \RZ + |k(0)| \\
&\leq \LC  \int_{0}^{1} \| \pi (s x) \| ^2 ds \RC ^{1/2}  \|  x \|  + |k(0)| \\
&\leq
\LC \LC  \int_{0}^{1} \| \pi (s x) - \pi (0)\| ^2 ds \RC ^{1/2}
+
  \LC  \int_{0}^{1} \|  \pi (0)\| ^2 ds \RC ^{1/2}
\RC \|  x \|  + |k(0)| \\
&\leq
\LC \frac{L}{2}   \|  x \| 
+ \|  \pi (0)\| 
\RC \|  x \|  + |k(0)| , 
\end{align*}
namely, 
\begin{equation}
|k (x) | 
\leq 
C _1 (\| x\| ^2 +1 )
\label{est-of-k} 
\end{equation}
holds for any $x \in \R ^  N$.
Therefore we obtain 
\begin{align*}
& \frac{1}{2q}  \| u _{\lambda }(t) \| ^q 
- C_1  (\| u _{\lambda }(t) \|  ^2 +1 )
\\
&
 \leq 
C_1  \| f(t) \| ^{q' }
+
\| f(0) \| \| u_ 0 \|
+ \int_{0}^{t} \| f'(s) \| \| u _\lambda (s) \| ds
+
\varphi   (u _0 )
+ \frac{1}{q} \| u _0 \| ^q 
+   k (u _ 0 ) . 
\end{align*}
Since $q > 2 $, 
\begin{equation*}
c _1 \|  u _{\lambda }(t) \| ^2 
\leq C _1 +  \int_{0}^{t} \| f'(s) \| \| u _\lambda (s) \| ds . 
\end{equation*}
Hence by Lemma \ref{Gronwall} (Gronwall's inequality). 
\begin{equation}
 \sup _{0\leq t \leq T  } \|  u _{\lambda }(t) \| 
\leq 
C_1  \LC 1 +  \int_{0}^{T} \| f'(t) \|  dt \RC . 
\label{uniform001}
\end{equation}
From \eqref{bounded}, 
\begin{equation*}
 \sup _{0\leq t \leq T  } \| \partial \varphi _{\lambda } (u _{\lambda } (t) ) \|
\leq 
 \sup _{0\leq t \leq T  }\| ( \partial \varphi _{G,p} ) ^{\circ} (u _{\lambda } (t) ) \|
\leq 
\kappa  \sup _{0\leq t \leq T  } \| u _{\lambda } (t) \| ^ {p-1} \leq C_1 . 
\end{equation*}
Therefore, we obtain by the second equation of  \eqref{EqLam}
\begin{equation}
 \sup _{0\leq t \leq T  } \|  \mu  _{\lambda }(t) \| 
\leq 
C _1 ,
\label{uniform002}
\end{equation}
and by the first equation of  \eqref{EqLam}
\begin{equation}
 \sup _{0\leq t \leq T  } \|  u'  _{\lambda }(t) \| 
\leq 
C _1 .
\label{uniform003}
\end{equation}
According to the uniform boundedness \eqref{uniform001}--\eqref{uniform003}, 
we can show that the solution can be globally extended until given $T > 0 $. 
Moreover, \eqref{bounded} yields 
\begin{equation}
 \sup _{0\leq t \leq T  } \| \partial \varphi _{\lambda } ( \mu  _{\lambda }(t) ) \| 
\leq 
C _1 .
\label{uniform004}
\end{equation}

By the uniform boundedness \eqref{uniform001}--\eqref{uniform004}, 
there exist
subsequences $ \{ u _{\lambda _n  } \} _ {n \in \N } ,  \{ \mu _{\lambda _n  } \} _ {n \in \N } $
and $u , \mu , \xi , \zeta : [0,T ] \to \R ^N $ such that 
\begin{align*}
& u _{\lambda _n } \to u  & \text{strongly in } C( [0,T] ; \R ^N ) ,\\
& u' _{\lambda _n } \rightharpoonup  u'  & \ast \text{-weakly in } L ^{\infty }
			( 0,T ; \R ^N ) , \\
& \mu _{\lambda _n } \rightharpoonup  \mu  & \ast \text{-weakly in } L ^{\infty }
			( 0,T ; \R ^N ) ,\\
& \partial \varphi _{\lambda _n } ( u  _{\lambda _n }(t) ) \rightharpoonup \xi
		& \ast \text{-weakly in } L ^{\infty }
			( 0,T ; \R ^N ) , \\
& \partial \varphi _{\lambda _n } ( \mu  _{\lambda _n }(t) ) \rightharpoonup  \zeta
		& \ast \text{-weakly in } L ^{\infty }
			( 0,T ; \R ^N ) .
\end{align*}
By virtue of Lemma \ref{demiclosedness}, 
we easily have $\xi (t) \in \partial \varphi _{G,p } (u (t)) $ for a.e. $t\in [0,T]$.
Then $ u, \mu $ and $\zeta $ satisfy 
\begin{equation*}
\begin{cases}
~~ u' (t) + \zeta (t) = 0 ,\\
~~\mu (t) =  \xi (t) + a \| u(t) \| ^{q-2} u(t)
			+ \pi (u(t)) - f(t)  .
\end{cases}
\end{equation*}
So if we can say that 
$\zeta (t) \in \partial \varphi _{G,p } (\mu (t)) $ holds for a.e. $t\in [0,T]$,
we conclude that the limit $(u , \mu )$ is a solution to the original problem \eqref{Eq} with $\tau = 0 $.

To see this, we multiply the second equation of \eqref{EqLam} with $\lambda = \lambda _ n $
by $u ' _{\lambda _n }$ and take the limit $n \to \infty $.
Then we have 
\begin{align*}
&\liminf _{n \to \infty } \int_{0}^{T} \mu _{\lambda _n } (t) \cdot u' _{\lambda _n} (t) dt\\
= ~ &  \liminf _{n \to \infty }
\int_{0}^{T} 
\LC 
\frac{d}{dt }\varphi _{\lambda _n } (u _{\lambda _n } (t)  )
+ \frac{a}{q}\frac{d}{dt } \| u _{\lambda _n } (t) \| ^q 
+ \frac{d}{dt } k (u _{\lambda _n} (t))
-u ' _{\lambda _n } (t) \cdot f(t)   
\RC dt \\
=~ &  \liminf _{n \to \infty }
\LC 
\varphi _{\lambda _n } (u _{\lambda _n } (T)  )
-\varphi _{\lambda _n } (u _0 )
+ \frac{a}{q}  \| u _{\lambda _n } (T) \| ^q 
- \frac{a}{q}  \| u _0 \| ^q  \right . \\
& \hspace{4cm} \left .  
+  k (u _{\lambda _n} (T))
-  k (u _0)
-\int_{0}^{T} 
 u ' _{\lambda _n } (t) \cdot f(t)  dt  
\RC  \\
\geq ~ 
& 
 \liminf _{n \to \infty }\varphi _{\lambda _n } (u _{\lambda _n } (T)  )
-\varphi _{G,p } (u _0 )
+ \frac{a}{q}  \| u  (T) \| ^q 
- \frac{a}{q}  \| u _0 \| ^q \\
& \hspace{4cm} 
+  k (u  (T))
-  k (u _0)
-\int_{0}^{T} 
 u ' (t) \cdot f(t)  dt  ,
\end{align*}
where we use \eqref{MY-01}. 
Let $J _{\lambda }$ by the resolvent of $\partial \varphi _{G,p }$. 
By \eqref{uniform004} and the definition of the Yosida approximation,
\begin{equation*}
\sup _{0 \leq t \leq T} \| J _{\lambda _n } u _{\lambda _n } (t) - u _{\lambda _n } (t) \|
\leq \lambda _n C . 
\end{equation*}
Hence 
\begin{align*}
& J _{\lambda _n }u _{\lambda _n } \to u  & \text{strongly in } C( [0,T] ; \R ^N ) .
\end{align*}
Moreover, 
by the continuity of $\varphi _{G ,p }  $ and  
\eqref{MY-01}, 
\begin{align*}
 \liminf _{n \to \infty }\varphi _{\lambda _n } (u _{\lambda _n } (T)  )
& \geq 
 \liminf _{n \to \infty }\varphi _{G, p } ( J _{\lambda _n }u _{\lambda _n } (T)  ) \\
& = 
 \varphi _{G, p } ( u (T)  ) .  
\end{align*}
Therefore 
\begin{align*}
&\liminf _{n \to \infty } \int_{0}^{T} \mu _{\lambda _n } (t) \cdot u' _{\lambda _n} (t) dt\\
\geq ~
& 
 \varphi _{G, p } ( u (T)  ) 
-\varphi _{G,p } (u _0 )
+ \frac{a}{q}  \| u  (T) \| ^q 
- \frac{a}{q}  \| u _0 \| ^q 
+  k (u  (T))
-  k (u _0)
-\int_{0}^{T} 
 u ' (t) \cdot f(t)  dt  \\
= ~ 
&
\int_{0}^{T}  
\LC
\frac{d}{dt}  \varphi _{G, p } ( u (t)  ) 
+ \frac{a}{q} \frac{d}{dt}   \| u  (t) \| ^q 
+  \frac{d}{dt}  k (u  (t))
-
 u ' (t) \cdot f(t) \RC dt  \\
= ~
&
\int_{0}^{T}  
u ' (t) 
\cdot 
\LC
\xi (t) 
+a \| u(t) \| ^{q-2} u(t) 
+\pi (u (t))
- f(t) \RC dt  
=
\int_{0}^{T}  
u ' (t) 
\cdot 
\mu (t) dt . 
\end{align*}
This implies that 
\begin{equation}
\begin{split}
&\limsup _{n \to \infty } \int_{0}^{T} J_{\lambda _n }\mu _{\lambda _n } (t) \cdot u '_{\lambda _n} (t) dt
\geq 
\liminf _{n \to \infty } \int_{0}^{T} J_{\lambda _n }\mu _{\lambda _n } (t) \cdot u' _{\lambda _n} (t) dt
\\
= ~
& 
\liminf _{n \to \infty } 
\LC
\int_{0}^{T} \mu _{\lambda _n } (t) \cdot u' _{\lambda _n} (t) dt
+
\int_{0}^{T} ( J_{\lambda _n }\mu _{\lambda _n } (t)
- \mu _{\lambda _n } (t) )  \cdot u' _{\lambda _n} (t) dt
\RC
 \\
\geq ~
&
\liminf _{n \to \infty } 
\LC
\int_{0}^{T} \mu _{\lambda _n } (t) \cdot u ' _{\lambda _n} (t) dt
- \lambda _n 
\int_{0}^{T} \| \partial \varphi _{\lambda _n } (\mu _{\lambda _n } (t))\|
 \|  u' _{\lambda _n} (t) \|  dt
\RC
 \\
= ~
&
\int_{0}^{T}  
u ' (t) 
\cdot 
\mu (t) dt . 
\end{split}
\label{convergence000}
\end{equation}
On the other hand, by \eqref{uniform004} and the definition of the Yosida approximation,
\begin{equation*}
\sup _{ 0 \leq t \leq T} \| J_{\lambda _n }\mu _{\lambda _n } (t) -
\mu _{\lambda _n } (t) \|
= 
\lambda _n \sup _{ 0 \leq t \leq T} \| \partial \varphi _{\lambda _n } ( \mu _{\lambda _n } (t) )
 \| \to 0 
\end{equation*}
as $n \to \infty $, which yields 
\begin{align*}
& J _{\lambda _n } \mu  _{\lambda _n } \rightharpoonup  \mu   & \ast\text{-weakly in } L ^{\infty}( 0,T ; \R ^N ) .
\end{align*}
Then by \eqref{convergence000}, 
\begin{align*}
\limsup _{n \to \infty } \int_{0}^{T} J_{\lambda _n }\mu _{\lambda _n } (t) \cdot 
\partial \varphi _{\lambda _n } ( \mu _{\lambda _n} (t) ) dt
& =
\limsup _{n \to \infty } \int_{0}^{T} J_{\lambda _n }\mu _{\lambda _n } (t) \cdot 
( - u' _{\lambda _n} (t) ) dt \\
& =
- \liminf _{n \to \infty } \int_{0}^{T} J_{\lambda _n }\mu _{\lambda _n } (t) \cdot 
 u' _{\lambda _n} (t)  dt \\
& \leq 
- \int_{0}^{T} \mu  (t) \cdot 
 u'  (t)  dt \\
& \leq 
 \int_{0}^{T} \mu  (t) \cdot 
 \zeta  (t)  dt . 
\end{align*}
Since 
$\partial \varphi _{\lambda _n } ( \mu _{\lambda _n} (t) ) 
\in \partial \varphi _{G,p  } (J _{\lambda _n } \mu _{\lambda _n} (t) )$
for a.e. $t \in [0 , T]$, 
Lemma \ref{wwlimsup} implies that 
\begin{equation*}
\zeta (t) \in \partial \varphi _{G,p } (\mu (t)) ~~~~~\text{a.e. } t\in [0,T] , 
\end{equation*}
whence follows Theorem \ref{TH01}. 
\end{proof}

\section{Case of  $\tau > 0 $}

The aim of this section is to consider the case of $\tau >0 $
and show the following result:
\begin{Th}
\label{Th2}
For any $\tau > 0 $, $u _ 0 \in \R ^N $,  and $ f \in L ^2 ( 0, T ; \R ^N )$, 
\eqref{Eq} possesses at least one solution satisfying 
\begin{equation*}
u \in W ^{1,2 } (0,T ; \R ^N ), ~~~~~\mu \in L ^{2} (0,T ; \R ^N ). 
\end{equation*}
\end{Th}
To this end, we first deal with the following equations 
obtained by replacing the hypergraph Laplacian $\mathcal{L} _{G,p }$
 with its Yosida approximation $\partial \varphi _{\lambda }$: 
\begin{equation}
\begin{cases}
~~ u' (t) + \partial \varphi _{\lambda } (\mu (t) ) = 0 
			~~&~~t \in (0,T) , \\
~~\mu (t) = \tau u' (t) + \partial \varphi _{\lambda } (u(t)) + a \| u(t) \| ^{q-2} u(t)
			+ \pi (u(t)) - f(t) ~~&~~ t \in (0,T) , \\
~~u(0) = u_0 ,
\end{cases}
\label{EqLam_tau}
\tag{P$_  { \lambda , \tau }$}
\end{equation}
and prove the following:
\begin{Th}
\label{Th3} 
For any $ \tau > 0 $, $\lambda >0 $,  $u _ 0 \in \R ^N $, and $f \in L ^2 (0,T ;\R ^N )$, 
\eqref{EqLam_tau} possesses a unique solution satisfying 
\begin{equation*}
u \in W ^{1,2 } (0,T ; \R ^N ) ,  ~~~\mu \in L^2 (0,T ; \R ^N ) . 
\end{equation*}
\end{Th}

\begin{proof}[Proof of Theorem \ref{Th3}]
Define a mapping 
$ \mathscr{T} : L ^{\infty } (0, T' ; \R ^ N )
\to  L ^{\infty } (0, T' ; \R ^ N )$
in the following way: 
\begin{itemize}
\item[i)] Fix $ \mu \in  L ^{\infty } (0, T' ; \R ^ N )$ and define 
$u _{\mu } \in W ^{1, \infty } (0, T' ; \R ^ N )$ by 
\begin{equation*}
\begin{cases*}
~~ u' _{\mu } (t) = -  \partial \varphi _{\lambda } (\mu (t) ), \\
~~ u_{\mu}(0) = u_0. 
\end{cases*}
\end{equation*}

\item[ii)] By using $u _ {\mu } $ in i), we 
define $\mu _u \in L ^{\infty } (0, T' ; \R ^ N )$
by 
\begin{equation*}
\mu _ u  (t) = \tau u'_{\mu} (t) + \partial \varphi _{\lambda } (u _{\mu }(t))
		+ a \| u_{\mu }(t) \| ^{q-2} u_{\mu }(t)
			+ \pi (u_{\mu } (t)) - f(t) .
\end{equation*}
To define $\mu _u \in L ^{\infty } (0, T' ; \R ^ N )$, 
we here assume $ f \in L ^{\infty } (0, T' ; \R ^ N ) $. 

\item[iii)] Define 
$ \mathscr{T} : L ^{\infty } (0, T' ; \R ^ N )
\to  L ^{\infty } (0, T' ; \R ^ N )$ by 
$ \mathscr{T} (\mu) :=  \mu _ u $.

\end{itemize}
We define $K _M $ by \eqref{KM}. 
When $\mu \in K_M $, we get by the first equation of \eqref{EqLam_tau}
\begin{equation*}
\sup _{0 \leq t \leq T' } 
\| u ' _{\mu } (t) \|
\leq
 \sup _{0 \leq t \leq T' } 
 \|  (\partial \varphi _{G,p }  ) ^{\circ}(\mu (t))  \|
 \leq
\kappa  \sup _{0 \leq t \leq T' } 
 \| \mu (t) \| ^{p-1}
 \leq 
\kappa  M ^{p-1} , 
\end{equation*}
where we used \eqref{bounded} and \eqref{MY-02}.
Then 
\begin{equation*}
\sup _{0 \leq t \leq T' } \| u _{\mu} (t) \|
\leq \| u_0 \| + \kappa T' M ^{p-1}  . 
\end{equation*}
By the second equation of \eqref{EqLam_tau}, 
\begin{align*}
\| \mu _u (t) \| 
&\leq \tau \| u' _{\mu} (t)\| + \| \partial \varphi _{\lambda } (u _{\mu } )\|
+a \| u_{\mu} (t) \| ^{q-1} + \| \pi (u _{\mu} (t) )\| + \| f (t) \| \\
&\leq \tau \kappa  M ^{p-1} 
+\kappa ( \| u_0 \| + \kappa T' M ^{p-1} ) ^{p-1}
+a ( \| u_0 \| + \kappa T' M ^{p-1} ) ^{q-1} \\
& \hspace{3cm} +L ( \| u_0 \| + \kappa T' M ^{p-1} )
+ \| \pi (0) \| + \| f (t) \| .
\end{align*}
Hence by fixing 
\begin{align*}
M& := 
 1
 +\kappa ( \| u_0 \| + 1 ) ^{p-1}
+a ( \| u_0 \| + 1 ) ^{q-1} + L ( \| u_0 \| + 1 )
+ \| \pi (0) \| + \sup _{0\leq t \leq T } \| f (t) \| ,\\
T' & \leq 1/ \kappa M ^{p-1} , \\
\tau & \leq  1/ \kappa  M ^{p-1} , 
\end{align*}
we obtain 
\begin{equation*}
\sup _{0\leq t \leq T'} \| \mu _u (t) \| 
\leq M 
\end{equation*}
and $  \mathscr{T} (  K_M  ) \subset K_M $. 
We also have 
\begin{equation*}
\sup _{0 \leq t \leq T' } \| u _{\mu} (t) \|
\leq \| u_0 \| + 1
\end{equation*}
for any $ \mu \in K _M  $.

Let $ \mu _1 , \mu _2 \in K _M $ and $ u ' _ {\mu _i } =  - \partial \varphi _{\lambda }  (\mu _i )$ ($i=1,2 $).  
By the first  equation of  \eqref{EqLam_tau}, 
\begin{equation}
\sup _{ 0\leq t \leq T' }\|  u ' _ {\mu _1 }  (t)- u ' _ {\mu _2 } (t)\|  =
\sup _{ 0\leq t \leq T' } \| \partial \varphi _{\lambda }  (\mu _1 (t)) - \partial \varphi _{\lambda }  (\mu _2 (t))  \|
\leq 
\frac{1}{\lambda }  \sup _{ 0\leq t \leq T' } \| \mu _1  (t)- \mu _2  (t) \|
\label{est-difference-tau00}
\end{equation}
which yields
\begin{align}
\sup _{ 0\leq t \leq T' } \| u  _{ \mu _1 } (t) - u _{\mu _2 } (t) \|
\leq \frac{T' }{\lambda } \sup _{ 0\leq t \leq T' } \| \mu _1 (t) - \mu _ 2 (t) \|  .
\label{est-difference-tau01}
\end{align}
By the second equation of \eqref{EqLam_tau},
$\mu _ { u _i } =  \mathscr{T} ( \mu _i ) $ ($i=1,2 $) satisfy 
\begin{align*}
\| \mu _{u_1} (t) - \mu _{u_2} (t) \|
\leq &
\tau \|  u' _{\mu _1 } (t) - u ' _{\mu _2 } (t) \| 
+ \frac{1}{\lambda } \| u_{\mu _1 } (t) - u_{\mu _2 } (t) \| 
\\
& +a \LN \| u_{\mu _1} (t) \| ^{q-2} u_{\mu _1}(t) -\| u_{\mu _2} (t) \| ^{q-2} u_{\mu _2} (t) \RN 
+
L \| u_{\mu _1} (t) -u_{\mu _2} (t) \| , 
\end{align*}
where $L $ is a Lipschitz constant of $\pi : \R ^ N \to \R ^N $
and we use \eqref{Lipschitz}. 
By  the Taylor expansion, there exists some constant $ \ell (\| u_0 \|)$ depending only on $\| u_0 \|$
such that 
\begin{align*}
\LN \| u_{\mu _1} (t) \| ^{q-2} u_{\mu _1} (t) -\| u_{\mu _2}(t)  \| ^{q-2} u_{\mu _2} (t) \RN 
 \leq \ell (\| u_0 \|)
\LN u_{\mu _1} (t) - u_{\mu _2} (t)  \RN .
\end{align*}
Hence by \eqref{est-difference-tau00} and \eqref{est-difference-tau01}, 
\begin{align*}
& \sup _{0 \leq t \leq T' } \| \mu _{u_1} (t) - \mu _{u_2} (t) \| \\
\leq ~ & 
 \tau  \sup _{0 \leq t \leq T' }  \| u' _{\mu _1 } (t) - u'_{\mu _2 } (t) \|
+
\frac{1}{\lambda } \sup _{0 \leq t \leq T' }  \| u_{\mu _1 } (t) - u_{\mu _2 } (t) \|
 \\
&\hspace{1cm} +a\ell (\| u_0  \|) \sup _{0 \leq t \leq T' } \LN u_{\mu _1} (t) - u_{\mu _2} (t)  \RN
+
L \sup _{0 \leq t \leq T' } \| u_{\mu _1} (t) -u_{\mu _2} (t) \| \\
\leq ~&
\frac{\tau }{\lambda }  \sup _{0 \leq t \leq T' }  \| \mu _1  (t) - \mu _2  (t) \|
+
\frac{T'}{\lambda ^2 } \sup _{0 \leq t \leq T' }  \| \mu _1  (t) - \mu _2  (t) \|
 \\ 
&\hspace{1cm} +
\frac{aT'\ell (\| u_0  \|)}{\lambda }  \sup _{0 \leq t \leq T' } \| \mu _1 (t)-\mu _2 (t)  \|
+
\frac{LT'}{\lambda } \sup _{0 \leq t \leq T' } \| \mu _1 (t) -\mu _2 (t) \|  .
\end{align*}
Now by fixing $\tau > 0 $ and $T '> 0$ by  
\begin{equation}
T':= 
\min \LD 
\frac{1}{2} 
\LC
\frac{1}{\lambda ^2 }
+
\frac{a\ell (\| u_0  \|)}{\lambda }
+
\frac{L}{\lambda } \RC ^{-1} , 
\frac{1 }{	\kappa M ^{p-1}  } 
\RD , 
~~~
\tau \leq \frac{1}{4} \lambda  ,
\label{tau-existence}
\end{equation}
we obtain 
\begin{equation*}
\sup _{0 \leq t \leq T' } \| \mu _{u_1} (t) - \mu _{u_2} (t) \|
\leq 
\frac{3}{4}  \sup _{0 \leq t \leq T' } \| \mu _1 (t) -\mu _2 (t) \| , 
\end{equation*}
which implies that $ \mathscr{T}  $ is a contraction map on $K_M $
and \eqref{EqLam} has a time local solution between $(0,T' )$
for any $f \in L ^{\infty } (0,T' ; \R ^N )$.

According to \eqref{tau-existence}, we remark that $\tau $ depends on $\lambda $.
Then if we simply let $\lambda \to 0$, then $\tau \to 0$, and the term $\tau u'$ disappears.
Therefore, we need to show that \eqref{EqLam_tau} has a solution for any $\tau > 0$ independent of $\lambda $.

To this end, we here extend the external force $f $ to $L ^2 (0,T ; \R ^N )$.
Let $\{ f _ n \} _{n\in \N } $ be $f _n \in L ^{\infty } (0,T ; \R ^N )$ and 
$f _n \to f $ strongly in $L ^2 (0,T ; \R ^N )$. 
Furthermore, let $ (u _ n , \mu _n )$ be a solution to 
\begin{equation}
\begin{cases}
~~ u' _n (t) + \partial \varphi _{\lambda } (\mu _n (t) ) = 0 
			, \\
~~\mu _n (t) = \tau u' _n  (t) + \partial \varphi _{\lambda } (u_n (t)) + a \| u _n  (t) \| ^{q-2} u_n  (t)
			+ \pi (u_n  (t)) - f_n  (t) , \\
~~u _n  (0) = u_0 . 
\end{cases}
\label{EqLam_n}
\tag{P$_  { n }$}
\end{equation}
Multiplying the first equation of \eqref{EqLam_n} by $\mu _n (t) $
and the second equation by $ u ' _n (t)$, we have 
\begin{align*}
 u ' _n (t) \cdot \mu _n (t) &= - \partial \varphi _{\lambda } (\mu _n (t) ) \cdot \mu _n (t) 
\leq  0 , \\
 \mu _n \cdot u ' _n (t) &= \tau \|  u' _n (t)\| ^2  + \frac{d}{dt} \LC  \varphi _{\lambda } (u _n (t) ) + \frac{a}{q} \| u _ n(t)\| ^q + k (u _n (t))
 \RC  +u  ' _n (t )\cdot f_ n (t) \\
 &\geq  \frac{ \tau}{2 } \|  u' _n (t)\| ^2 - \frac{1}{2 \tau }  \| f _ n (t)\| ^2
		+ \frac{d}{dt} \LC  \varphi _{\lambda } (u _n (t) ) + \frac{a}{q} \| u _ n(t)\| ^q + k (u _n (t)) \RC ,
\end{align*} 
which yields 
\begin{equation*}
 \frac{ \tau}{2 } \|  u' _n (t)\| ^2  
		+ \frac{d}{dt} \LC  \varphi _{\lambda } (u _n (t) ) + \frac{a}{q} \| u _ n(t)\| ^q + k (u _n (t)) \RC
		\leq \frac{1}{2 \tau }  \| f_ n (t)\| ^2 . 
\end{equation*}
Integrating this, we get 
\begin{align*}
& \frac{ \tau}{2 } \int_{0}^{T}  \|  u' _n (t)\| ^2 dt   
		+  \LC  \varphi _{\lambda } (u _n (T) ) + \frac{a}{q} \| u _ n(T )\| ^q + k (u _n (T)) \RC \\
	&	\leq \frac{1}{2 \tau } \int_{0}^{T}  \| f_ n (t)\| ^2 
		+ \LC  \varphi _{\lambda } (u _0  ) + \frac{a}{q} \| u _0 \| ^q + k (u _0 ) \RC . 
\end{align*}
By  $\varphi _\lambda \geq 0 $,  \eqref{est-of-k}, and $q > 2 $, 
we obtain 
\begin{equation}
\int_{0}^{T}  \|  u' _n (t)\| ^2 dt    \leq C_2 , 
\label{est-f-L2-001}
\end{equation}
where and henceforth $C_2 > 0 $ is a general constant independent of $n \in \N$. 
Hence 
\begin{equation}
\sup _{0 \leq t \leq T }  \|  u _n (t)\| ^2 dt    \leq C_2 
\label{est-f-L2-002}
\end{equation}
and from the second equation of \eqref{EqLam_n}
\begin{equation}
\int_{0}^{T}   \|  \mu _n (t)\| ^2 dt    \leq C_2 .
\label{est-f-L2-003}
\end{equation}
The estimates \eqref{est-f-L2-001}--\eqref{est-f-L2-003}
imply that the solution to \eqref{EqLam_n} dose not blow-up 
and can be extended globally up to $[0,T ]$.

By a priori estimates \eqref{est-f-L2-001}--\eqref{est-f-L2-003},
there exist some subsequence of $ \{ u _n  \} _ {n \in \N } ,  \{ \mu _n  \} _ {n \in \N } $ (we omit relabeling)
and the limits $u , \mu  : [0,T ] \to \R ^N $ such that 
\begin{align*}
& u _n \to u  & \text{strongly in } C( [0,T] ; \R ^N ) ,\\
& u' _n \rightharpoonup  u'  & \text{weakly in } L ^{2}
			( 0,T; \R ^N ) , \\
& \mu _ n  \rightharpoonup  \mu  & \text{weakly in } L ^{2 }
			( 0,T ; \R ^N ) .
\end{align*}
Moreover, from 
\eqref{est-f-L2-002} and \eqref{bounded}, 
\begin{equation*}
\sup _{0 \leq t \leq T }   \|  \partial \varphi _{\lambda } (u _ n (t) )\| \leq C_2  
\end{equation*}
and from \eqref{est-f-L2-002} and the first equation of \eqref{EqLam_n}, 
\begin{equation*}
\int_{0}^{T}    \|  \partial \varphi _{\lambda } (\mu _ n (t) )\| ^2 dt  \leq C_2  .
\end{equation*}
Hence there are some $\xi , \zeta : [0,T ] \to \R ^N $ satisfying 
\begin{align*}
& \partial \varphi _{\lambda } ( u  _n (t) ) \rightharpoonup \xi
		& \ast \text{-weakly in } L ^{\infty }
			( 0,T ; \R ^N ) , \\
& \partial \varphi _{\lambda _n } ( \mu  _n (t) ) \rightharpoonup  \zeta
		& \text{weakly in } L ^{2 }
			( 0,T ; \R ^N ) .
\end{align*}
By the demiclosedness of subdifferentials (Lemma \ref{demiclosedness}), 
we get $\xi (t) = \partial \varphi _{\lambda  } (u (t)) $ for a.e. $t\in [0,T]$.
Then $ u, \mu $ and $\zeta $ satisfy 
\begin{equation*}
\begin{cases}
~~ u' (t) + \zeta (t) = 0 ,\\
~~\mu (t) = \tau u' (t) +  \partial \varphi _{\lambda  } (u (t))  + a \| u(t) \| ^{q-2} u(t)
			+ \pi (u(t)) - f(t) .
\end{cases}
\end{equation*}
Therefore 
if $ \zeta (t) \in \partial \varphi _{ \lambda  } ( \mu (t)) $ for a.e. $t \in [0,T ]$ is valid, 
we can show that \eqref{EqLam_tau} has a solution for  $f\in L^2 (0,T ;\R ^N )$.

Multiplying the second equation of \eqref{EqLam_n} by $u ' _n $
and integrating over $[0,T ]$, we have 
\begin{align*}
& \liminf _{n \to \infty } \int_{0}^{T} u' _n (t) \cdot \mu _n (t) dt \\
  ={} &  \liminf _{n\to \infty } 
\LC  \tau \int_{0}^{T} \| u ' _n (t)\| ^2 dt + \varphi _{\lambda } (u _ n (T) ) -
\varphi _{\lambda } (u _0  )  \right .  \\
& \hspace{2cm} \left . 
 {} + {} 
\frac{a}{q} \LC \| u _n (T) \| ^q - \| u _0 \| ^q \RC + k (u _ n (T)) - k ( u_ 0 )
 - \int_{0}^{T} f_n (t ) \cdot u' _n (t)  dt   \RC \\
  \geq {} &
\tau \int_{0}^{T} \| u ' (t)\| ^2 dt + \varphi _{\lambda } (u  (T) ) -
\varphi _{\lambda } (u _0  )    \\
&
\hspace{2cm}  
 {} + {} 
\frac{a}{q} \LC \| u  (T) \| ^q - \| u _0 \| ^q \RC + k (u  (T)) - k ( u_ 0 )
 - \int_{0}^{T} f (t ) \cdot u'  (t)  dt   \\
= &  {} 
 \int_{0}^{T} \LC \tau \| u ' (t)\| ^2 + \frac{d}{dt} \varphi _{\lambda } (u (t))  
			+ \frac{a}{q} \frac{d}{dt}  \| u  (t) \| ^q  +  \frac{d}{dt}  k (u  (t)) 
			 - f(t) \cdot u' (t) \RC dt 
    \\
= & {}  
 \int_{0}^{T} \LC \tau u ' (t)+ \partial  \varphi _{\lambda } (u (t))  
			+ a\| u  (t) \| ^{q-2} u(t)  +  \pi (u  (t)) 
			 - f(t) \RC \cdot u' (t) dt 
    \\
= &   {}  
 \int_{0}^{T} \mu (t) \cdot u' (t) dt .
\end{align*}
Namely, 
\begin{equation*}
\limsup _{n \to \infty }  \int_{0}^{T}  \mu _n (t) \cdot (- u' _n (t) ) dt
\leq 
 \int_{0}^{T} \mu (t) \cdot ( - u' (t) )  dt
= 
 \int_{0}^{T} \mu (t) \cdot \zeta (t )  dt , 
\end{equation*}
which together with Lemma \ref{wwlimsup} yields
$\zeta (t) = \partial  \varphi _{\lambda } (\mu (t))$ for a.e. $t \in [0,T]$
and then 
\eqref{EqLam_tau} possesses a solution for  $f\in L^2 (0,T ;\R ^N )$.

We next show the uniqueness of solution. Let $(u _ i , \mu _ i )$ ($i=1,2 $) 
be solutions to 
\begin{equation*}
\begin{cases}
~~ u' _ i  (t) + \partial \varphi _{\lambda } (\mu _ i (t) ) = 0 
			~~&~~t \in (0,T) , \\
~~\mu _ i (t) = \tau u' _ i (t) + \partial \varphi _{\lambda } (u _ i (t)) + a \| u _ i (t) \| ^{q-2} u _ i (t)
			+ \pi (u _ i (t)) - f _ i (t) ~~&~~ t \in (0,T) , \\
~~u_ i (0) = u_{i0} .  
\end{cases}
\end{equation*}
Multiplying the first equation by $ \mu _ i $ and the second equation by $ u ' _i $, 
we get 
\begin{align*}
 \mu _i (t) \cdot u ' _ i (t) &= -  \partial \varphi _{\lambda } (\mu _ i (t) ) \cdot  \mu _i (t) \leq 0 ,  \\
  u ' _ i (t) \cdot \mu _i (t)  & \geq 
		\tau \| u' _ i (t)\| ^2 + \frac{d}{dt}  \varphi _{\lambda } (u _ i (t)) + \frac{a}{q}  \frac{d}{dt} \| u _ i (t) \| ^{q} 
			+ \frac{d}{dt} k  (u _ i (t)) - \| f _ i (t) \| \| u' _i (t ) \| \\ 
 & \geq 
		\frac{\tau }{2}  \| u' _ i (t)\| ^2 + \frac{d}{dt}  \varphi _{\lambda } (u _ i (t)) + \frac{a}{q}  \frac{d}{dt} \| u _ i (t) \| ^{q} 
			+ \frac{d}{dt} k  (u _ i (t)) -\frac{1}{2 \tau }  \| f _ i (t) \| ^2  ,
\end{align*}
which implies 
\begin{align*}
&\frac{\tau }{2} \int_{0}^{T}  \| u' _ i (t)\| ^2 dt + \varphi _{\lambda } (u _ i (T)) + \frac{a}{q} \| u _ i (T) \| ^{q} 
			+ k  (u _ i (T))  \\
& \hspace{1cm} \leq \frac{1}{2 \tau } \int_{0}^{T}   \| f _ i (t) \| ^2 dt 
			 +\varphi _{\lambda } (u _ {i0} ) + \frac{a}{q} \| u _ {i0}  \| ^{q} 
			+ k  (u _ {i0}  ) .
\end{align*}
Then there exists a constant $C (u_ {i0 }, f _ i )$ depending only on $u_ {i0 }$ and $f _ i $
such that 
\begin{equation*}
\int_{0}^{T}  \| u' _ i (t)\| ^2 dt  \leq C (u_ {i0 }, f _ i ) .
\end{equation*}
Immediately, 
\begin{equation}
\sup _{ 0\leq t \leq T } \| u _ i (t)\|  \leq C (u_ {i0 }, f _ i )
\label{uniquness0001}
\end{equation}
holds. 

Multiplying the difference of the first equations by $ \mu _ 1 - \mu _2 $
and multiplying 
the difference of the second  equations by $ u ' _1 - u' _2  $, 
we have 
\begin{align*}
 ( u ' _1 (t) - u ' _2 (t)  ) \cdot (\mu _1 (t) - \mu _2 (t) )
					&= - (\partial \varphi _{\lambda } ( \mu _1 (t)  )  - \partial \varphi _{\lambda } ( \mu _2 (t)  ) )	
							\cdot (\mu _1 (t) - \mu _2 (t) ) \leq 0 , \\ 
 (\mu _1 (t) - \mu _2 (t) ) \cdot  ( u ' _1 (t) - u ' _2 (t)  ) 
					& \geq \tau \| u' _1 (t)  - u' _2 (t)  \| ^2 - \frac{1}{\lambda } \| u_1 (t) - u_2 (t) \|  \| u' _1 (t)  - u' _2 (t)  \| \\
					 & - a \| \| u _1 (t) \| ^{q-2} u _1 (t)  - \| u _2 (t) \| ^{q-2} u _2 (t)   \|  \| u' _1 (t)  - u' _2 (t)  \| \\
					 &   - L \| u_1 (t) - u_2 (t) \|  \| u' _1 (t)  - u' _2 (t)  \| \\
					 &  -  \| f_1 (t) - f_2 (t) \|  \| u' _1 (t)  - u' _2 (t)  \|  .
\end{align*}
By the Taylor expansion and \eqref{uniquness0001}, 
there exists some constant $\ell (u _{10}, u _{2 0} , f_1 , f_2 ) $
depending only on $ u _{10}, u _{2 0} , f_1 , f_2 $ such that 
\begin{equation*}
\| \| u _1 (t) \| ^{q-2} u _1 (t)  - \| u _2 (t) \| ^{q-2} u _2 (t)   \| 
\leq 
\ell (u _{10}, u _{2 0} , f_1 , f_2 ) \|  u _1 (t) -  u _2 (t) \| . 
\end{equation*}
Hence we obtain 
\begin{align*}
\frac{\tau}{2}  \| u' _1 (t)  - u' _2 (t)  \| ^2
					& \leq  \frac{2}{\lambda \tau } \| u_1 (t) - u_2 (t) \| ^2
							+ \frac{2 a^2 \ell (u _{10}, u _{2 0} , f_1 , f_2 ) ^2 }{\tau }  \|  u _1 (t) -  u _2 (t) \| ^2  \\
					 & \hspace{1cm}   \frac{2 L ^2 }{\tau } \| u_1 (t) - u_2 (t) \| ^2  + \frac{2}{\tau }  \| f_1 (t) - f_2 (t) \| ^2 .
\end{align*}
By using this inequality, we deduce 
\begin{align*}
\frac{d}{dt} \| u _1 (t)  - u _2 (t)  \| ^2 &\leq 2 \| u _1 (t)  - u _2 (t) \|   \| u' _1 (t)  - u ' _2 (t) \| \\
		& \leq 
		 \frac{2}{\lambda ^ {1/2 } \tau } \| u_1 (t) - u_2 (t) \| ^2
							+ \frac{2 a \ell (u _{10}, u _{2 0} , f_1 , f_2 )  }{\tau }  \|  u _1 (t) -  u _2 (t) \| ^2  \\
					 & \hspace{1cm}   \frac{2 L }{\tau } \| u_1 (t) - u_2 (t) \| ^2  + \frac{2}{\tau }  \| f_1 (t) - f_2 (t) \|  \|  u _1 (t) -  u _2 (t) \| \\
		& \leq 
		\frac{2}{\tau } \LC 	\frac{1}{\lambda ^ {1/2 }  }  +   a \ell (u _{10}, u _{2 0} , f_1 , f_2 )  +  L + 1 \RC   \| u_1 (t) - u_2 (t) \| ^2 \\
		& \hspace{1cm}  + \frac{1}{2\tau }    \| f_1 (t) - f_2 (t) \| ^2  .
\end{align*}
With the aid of Gronwall's inequality, 
\begin{equation}
\sup _{0\leq t \leq T }   \| u_1 (t) - u_2 (t) \| ^2 
		 \leq  e ^{ C_3T} \| u_{10}  - u_{20 } \| ^2 
		+  \frac{e ^{C_3T} }{2\tau }   \int_{0}^{T}  \| f_1 (t) - f_2 (t) \| ^2 dt  ,	 
\label{est-diff-tau-001} 
\end{equation}
where 
\begin{equation*}
C_3 := \frac{2}{\tau } \LC 	\frac{1}{\lambda ^ {1/2 }  }  +   a \ell (u _{10}, u _{2 0} , f_1 , f_2 )  +  L + 1 \RC  .
\end{equation*}
This implies that the solution to \eqref{EqLam_tau} is unique.

Finally, we show that \eqref{EqLam_tau} has a solution for every $\tau > 0 $. 
Define 
\begin{equation*}
\mathcal{M} := \{ \tau _0  \in [0 , \infty ) ; ~~\text{\eqref{EqLam_tau} possesses a unique global solution for every } 0 \leq \tau \leq \tau _ 0 \} . 
\end{equation*}
By \eqref{tau-existence}, we can see that $ [ 0 , \lambda /  4  ] \subset \mathcal{M }$
and then $\mathcal{M }$ is not an empty set. 
We are going to show 
that $\mathcal{M}$ is open and closed in the relative topology induced by $\R $ on $[0, \infty )$, 
which implies $\mathcal{M} = [0, \infty ) $, in particular, \eqref{EqLam_tau} has a solution for arbitrary $\tau > 0 $.

We first show that $\mathcal{M}$ is closed. 
Let $\{ \tau _m \} \subset \mathcal{M}$ and $\tau _m \to \tau  _ 0 $ as $m \to \infty $.
If $ \tau _0 < \tau _ m $ with some $m \in \N $, 
then $\tau _0 \in \mathcal{M}$ immediately holds by the definition of $\mathcal{M}$.
Hence we can assume that $\tau _ m < \tau _0 $ for every $m \in \N $ without loss of generality.
Suppose that $ (u _ m , \mu _ m )  $ is a unique solution to 
\begin{equation}
\begin{cases}
~~ u' _m  (t) + \partial \varphi _{\lambda } (\mu _m  (t) ) = 0 
			 , \\
~~\mu _m  (t) = \tau _m  u' _m   (t) + \partial \varphi _{\lambda } (u_m  (t)) + a \| u _m   (t) \| ^{q-2} u_m   (t)
			+ \pi (u_m   (t)) - f  (t) , \\
~~u _m  (0) = u_0 ,  
\end{cases}
\label{EqLam_m}
\tag{P$_  { m }$}
\end{equation}
Multiplying the first equation of \eqref{EqLam_m} by $\mu _m $ and 
the second equation of \eqref{EqLam_m} by $u' _m $, we have 
\begin{align*}
u' _m  (t) \cdot \mu _m (t) & = -  \partial \varphi _{\lambda } (\mu _m  (t) ) \cdot \mu _m (t) \leq  0 , \\ 
 \mu _m (t) \cdot u' _m  (t) 
& = \tau _m \| u' _m  (t)\|  ^2 + \frac{d}{dt}  \varphi _{\lambda } (u_m  (t)) + \frac{a}{q}  \frac{d}{dt}  \| u _m   (t) \| ^{q}
			+ \frac{d}{dt} k (u_m   (t)) - f  (t) \cdot u' _m (t) \\
& \geq  \frac{\tau _m}{2}  \| u' _m  (t)\|  ^2 + \frac{d}{dt}  \varphi _{\lambda } (u_m  (t)) + \frac{a}{q}  \frac{d}{dt}  \| u _m   (t) \| ^{q}
			+ \frac{d}{dt} k (u_m   (t)) - \frac{1}{\tau _ m } \| f  (t) \| ^2  .
\end{align*}
Hence 
\begin{equation*}
\frac{\tau _m}{2}  \| u' _m  (t)\|  ^2 + \frac{d}{dt}  \varphi _{\lambda } (u_m  (t)) + \frac{a}{q}  \frac{d}{dt}  \| u _m   (t) \| ^{q}
			+ \frac{d}{dt} k (u_m   (t)) \leq  \frac{1}{\tau _ m } \| f  (t) \| ^2 , 
\end{equation*}
namely,  
\begin{equation}
\frac{\tau _m}{2} \int_{0}^{T}  \| u' _m  (t)\|  ^2 dt 
		\leq  \frac{1}{\tau _ m } \int_{0}^{T}  \| f  (t) \| ^2 dt   +C _4 , 
\label{est-closed-001}
\end{equation}
where and henceforth, $C_4 > 0 $ is a general constant independent of $ m $. 
Since we assume that 
$\tau _ m < \tau _0 $ and $\tau _m \to \tau _ 0 $, 
we obtain $\tau _ 0 > 0 $ and 
there exists some large $m_ 0 \in \N $ such that 
$ 0 < \tau _0 /2  < \tau _m $ for any $ m \geq m_0 $.  
Therefore, \eqref{est-closed-001} yields 
\begin{equation}
 \int_{0}^{T}  \| u' _m  (t)\|  ^2 dt 
		\leq  \frac{2}{\tau ^2 _ m } \int_{0}^{T}  \| f  (t) \| ^2 dt   +\frac{2}{\tau _ m }  C _4 \leq C_4 
\label{est-closed-002}		
\end{equation}
for every $m \geq m_0 $. 
This immediately implies 
\begin{equation}
\sup _{ 0 \leq t \leq T }  \| u _m  (t)\|  ^2 
		\leq   C_4 .
\label{est-closed-003}		
\end{equation}
From the second equation of \eqref{EqLam_m}, we can derive 
\begin{equation}
 \int_{0}^{T}  \| \mu  _m  (t)\|  ^2 dt 
		\leq  C_4 . 
\label{est-closed-004}		
\end{equation}
By \eqref{est-closed-002}--\eqref{est-closed-004}, there exist
some subsequences of $ \{ u _m  \} _ {m \in \N } ,  \{ \mu _m  \} _ {m \in \N } $ (we omit relabeling)
and $u , \mu  : [0,T ] \to \R ^N $ such that 
\begin{align*}
& u _m \to u  & \text{strongly in } C( [0,T] ; \R ^N ) ,\\
& u' _m \rightharpoonup  u'  & \text{weakly in } L ^{2}
			( 0,T; \R ^N ) , \\
& \mu _ m  \rightharpoonup  \mu  & \text{weakly in } L ^{2 }
			( 0,T ; \R ^N ) .
\end{align*}
We also have 
\begin{align*}
& \tau _ m u' _m \rightharpoonup \tau _0  u'  & \text{weakly in } L ^{2}
			( 0,T; \R ^N ) .
\end{align*}
Moreover,  by  \eqref{est-closed-003} and \eqref{bounded}, we have 
\begin{equation*}
\sup _{0 \leq t \leq T }   \|  \partial \varphi _{\lambda } (u _ m (t) )\| \leq C_4  
\end{equation*}
and by \eqref{est-closed-002} and the first equation of \eqref{EqLam_m}, we get 
\begin{equation*}
\int_{0}^{T}    \|  \partial \varphi _{\lambda } (\mu _ m (t) )\| ^2 dt  \leq C_4 . 
\end{equation*}
Hence there exist 
$\xi , \zeta : [0,T ] \to \R ^N $ such that 
\begin{align*}
& \partial \varphi _{\lambda } ( u  _m (t) ) \rightharpoonup \xi
		& \ast \text{-weakly in } L ^{\infty }
			( 0,T ; \R ^N ) , \\
& \partial \varphi _{\lambda _n } ( \mu  _m (t) ) \rightharpoonup  \zeta
		& \text{weakly in } L ^{2 }
			( 0,T ; \R ^N ) .
\end{align*}
By Lemma \ref{demiclosedness}, 
we get $\xi (t) = \partial \varphi _{\lambda  } (u (t)) $ for a.e. $t\in [0,T]$.
Then $ u, \mu $ and $\zeta $ satisfy 
\begin{equation*}
\begin{cases}
~~ u' (t) + \zeta (t) = 0 ,\\
~~\mu (t) = \tau _0 u' (t) +  \partial \varphi _{\lambda  } (u (t))  + a \| u(t) \| ^{q-2} u(t)
			+ \pi (u(t)) - f(t)  . 
\end{cases}
\end{equation*}
Multiplying the second equation of \eqref{EqLam_m} by $u ' _m $,
we get 
\begin{align*}
& \liminf _{m \to \infty } \int_{0}^{T} u' _m (t) \cdot \mu _m (t) dt \\   
 = & ~  \liminf _{m\to \infty } 
\LC  \tau _m \int_{0}^{T} \| u ' _m (t)\| ^2 dt + \varphi _{\lambda } (u _ m (T) ) -
\varphi _{\lambda } (u _0  )  \right .  \\
& \hspace{2cm} \left . 
+ 
\frac{a}{q} \LC \| u _m (T) \| ^q - \| u _0 \| ^q \RC + k (u _ m (T)) - k ( u_ 0 )
 - \int_{0}^{T} f (t ) \cdot u' _m (t)  dt   \RC \\
  \geq & ~  
\tau _0 \int_{0}^{T} \| u ' (t)\| ^2 dt + \varphi _{\lambda } (u  (T) ) -
\varphi _{\lambda } (u _0  )    \\
&\hspace{2cm} 
+ 
\frac{a}{q} \LC \| u  (T) \| ^q - \| u _0 \| ^q \RC + k (u  (T)) - k ( u_ 0 )
 - \int_{0}^{T} f (t ) \cdot u'  (t)  dt   \\
= & ~ 
 \int_{0}^{T} \LC \tau _0  u ' (t)+ \partial  \varphi _{\lambda } (u (t))  
			+ a\| u  (t) \| ^{q-2} u(t)  +  \pi (u  (t)) 
			 - f(t) \RC \cdot u' (t) dt 
    \\
 = & ~  
 \int_{0}^{T} \mu (t) \cdot u' (t) dt ,
\end{align*}
which implies 
\begin{equation*}
\limsup _{m \to \infty }  \int_{0}^{T}  \mu _m (t) \cdot (- u' _m (t) ) dt
\leq 
 \int_{0}^{T} \mu (t) \cdot ( - u' (t) )  dt
= 
 \int_{0}^{T} \mu (t) \cdot \zeta (t )  dt  . 
\end{equation*}
Hence $\zeta (t) = \partial  \varphi _{\lambda } (\mu (t))$ for a.e. $t \in [0,T]$
and we can assure that \eqref{EqLam_m} with $\tau  = \tau _ 0 $ has a solution. 
Therefore, $\tau _ 0 \in \mathcal{M}$ and  $\mathcal{M}  $ is a closed set.

Next we show that $\mathcal{M}$ is open. 
To this end, 
we check that for every $\tau _ 0 \in \mathcal{M}$ there exists some small $\delta > 0 $ such that 
$ [\tau _0 , \tau _0 + \delta ] \subset \mathcal{M}$.
Since $ [ 0 , \lambda /  4  ] \subset \mathcal{M }$, we can assume that 
$\tau _0 > 0 $ without loss of generality. 
For any given $h \in L ^2 (0,T ; \R ^N )$, let $(u_h , \mu _ h )$ be a unique solution to 
\begin{equation}
\begin{cases}
~~ u' _h   (t) + \partial \varphi _{\lambda } (\mu _h  (t) ) = 0 ,
			 \\
~~\mu  _h (t) = \tau  _0  u' _h  (t) + \partial \varphi _{\lambda } (u _h (t)) + a \| u   _h (t) \| ^{q-2} u_h   (t)
			+ \pi (u _h  (t)) - f  (t) + h (t ),  \\
~~u  _h (0) = u_0 .  
\end{cases}
\label{EqLam_open}
\tag{P$_  { h }$}
\end{equation}
In addition, we define a mapping $\mathscr{T} _{\delta } : L ^2 (0,T ; \R ^N ) \to L ^2 (0,T ; \R ^N ) $
by $\mathscr{T}_{\delta } (h) = \delta u' _h $, where $\delta >0 $ be fixed later. 
We here show that $\mathscr{T} _{\delta } $ is a contraction mapping on $ L ^2 (0,T ; \R ^N )$
for a sufficiently small $\delta $. 
Then we can see that there exists a solution 
to
\begin{equation*}
\begin{cases}
~~ u'     + \partial \varphi _{\lambda } (\mu    ) = 0 ,\\
~~\mu  = ( \tau _0 + \delta )   u'   + \partial \varphi _{\lambda } (u  ) + a \| u    \| ^{q-2} u  
			+ \pi (u ) - f  ,\\
~~u  (0) = u_0 ,  
\end{cases}
\end{equation*}
namely, we get $[\tau_ 0 , \tau _0 + \delta ] \subset \mathcal{M}$ and $\mathcal{M}$ is open. 
We define 
\begin{equation*}
K' _M := \LD 
h \in L ^2 (0,T ; \R ^N ) ; ~~\int_{0}^{T} \| h (t)\| ^2 dt \leq M  
\RD
\end{equation*}
for $M > 0 $. 
Multiplying the first equation of \eqref{EqLam_open} by $\mu  _ h $
and the second equation of  \eqref{EqLam_open} by $u ' _h $, 
we get 
\begin{align*}
& \tau _ 0 \| u' _h (t) \|  ^2 
+\frac{d}{dt}   \varphi _{\lambda } (u _h (t)) + \frac{a}{q} \frac{d}{dt}  \| u   _h (t) \| ^{q} 
			+\frac{d}{dt}  k (u _h  (t))  \\
\leq & ~ \| f(t)\| \| u' _h (t)\| +\| h (t)\| \| u' _h (t)\| \\
\leq & ~ \frac{1}{\tau _0 } \| f(t)\| ^2  +\frac{1}{\tau _0 } \| h (t)\| ^2 + \frac{\tau _0 }{2}   \| u' _h (t)\| ^2 
\end{align*}
Then by \eqref{est-of-k}, there exists some constant $ C_5 > 0  $ independent of $h $
such that 
\begin{align*}
&\int_{0}^{T}  \| u' _h (t) \|  ^2 dt  \\
& \leq 
\frac{2}{\tau _0 }
\LC  
\frac{1}{\tau  _0 }  \int_{0}^{T} \| f(t)\| ^2 dt  +\frac{1}{\tau  _0 }  \int_{0}^{T} \| h (t)\| ^2 dt 
+
 \varphi _{\lambda } (u _0 ) + \frac{a}{q}  \| u  _ 0  \| ^{q} 
			+   k (u _0) +C _5  \RC .
\end{align*}
Let $h \in K'_M $. From 
\begin{align*}
\int_{0}^{T}  \| \delta u' _h (t) \|  ^2 dt  
 \leq 
\frac{2 \delta ^2 }{\tau _0 }
\LC  
\frac{1}{\tau  _0 }  \int_{0}^{T} \| f(t)\| ^2 dt  +\frac{1}{\tau  _0 }  M
+
 \varphi _{\lambda } (u _0 ) + \frac{a}{q}  \| u  _ 0  \| ^{q} 
			+   k (u _0) +C _5  \RC , 
\end{align*}
we obtain $\delta u ' _ h \in K ' _M $ by letting 
$\delta > 0 $ be sufficiently small such that 
\begin{equation*}
\delta ^2 
			\leq  
			\frac{ \tau _0 M }{2}
\LC  
\frac{1}{\tau  _0 }  \int_{0}^{T} \| f(t)\| ^2 dt  +\frac{1}{\tau  _0 } M
+
 \varphi _{\lambda } (u _0 ) + \frac{a}{q}  \| u  _ 0  \| ^{q} 
			+   k (u _0) +C _5  \RC ^{-1} . 
\end{equation*}
Then for every $h \in K'_M $, we have 
\begin{align}
\sup _ {0\leq t \leq T} \| u _ h (t ) \| 
	\leq \int_{0}^{T} \|  u '_h  (t)\| dt + \| u_ 0 \|  
	\leq \frac{ T ^{1/2} M ^{1/2 } }{ \delta }  + \| u_ 0 \|  . 
\label{est-uh-open-001} 
\end{align}
Let $h _1 , h_2 \in K'_M$.
Multiplying the difference of 
the first equations of \eqref{EqLam_open} by $\mu  _ {h _1 } - \mu _{h _2 }  $
and the difference of second equations of  \eqref{EqLam_open} by $u ' _ { h_1 } -u ' _ { h_2 }  $, 
and repeating the same argument as above for the uniqueness of a solution to \eqref{EqLam_tau},
we obtain 
\begin{align*}
\tau _0  \| u' _{h_1}  (t)  - u' _{h_2}  (t)  \| ^2 
& \leq  \frac{1}{\lambda } \| u_{h_1}  (t) - u_{h_2} (t) \|  \| u' _{h_1}  (t)  - u' _{h_2} (t)  \| \\
					 & \hspace{1cm} + a \| \| u _{h_1} (t) \| ^{q-2} u _{h_1}  (t)  - \| u _{h_2} (t) \| ^{q-2} u _{h_2} (t)   \|  \| u' _{h_1}  (t)  - u' _{h_2} (t)  \| \\
					 & \hspace{1cm}  + L \| u_{h_1} (t) - u_{h_2} (t) \|  \| u' _{h_1}  (t)  - u' _{h_2} (t)  \| \\
					 & \hspace{1cm}  + \| h_1 (t) - h_2 (t) \|  \| u' _{h_1}  (t)  - u' _{h_2} (t)  \|  . 
\end{align*}
By the Taylor expansion and \eqref{est-uh-open-001}, 
there exists some constant $\ell (T, M , f, u _ 0  , \tau _ 0 ) $ depending only on 
$ T, M , f, u _ 0  , \tau _ 0  $ such that 
\begin{equation*}
\LN
\| u _{h_1 } (t) \| ^{q-2} u _{h_1}  (t)  - \| u _{h_2} (t) \| ^{q-2} u _{h_2} (t)   \RN
\leq 
\ell (T, M , f, u _ 0  , \tau _ 0 ) \| u _{h_1}  (t)  - u _{h_2} (t)  \| . 
\end{equation*}
Hence we obtain 
\begin{equation}
\begin{split}
& \frac{\tau _0 }{2}   \| u' _{h_1}  (t)  - u' _{h_2}  (t)  \| ^2 \\
					& \leq  \frac{2}{\lambda \tau _0  } \| u _{h_1}  (t)  - u _{h_2}  (t) \| ^2
							+ \frac{2 a^2 \ell (T, M , f, u _ 0  , \tau _ 0 ) ^2 }{\tau _0  }  \|  u _{h_1}  (t)  - u _{h_2}  (t) \| ^2  \\
					 & \hspace{2cm}  + \frac{2 L ^2 }{\tau _0 } \| u _{h_1}  (t)  - u _{h_2}  (t) \| ^2  + \frac{2}{\tau_0  }  \| h_1 (t) - h_2 (t) \| ^2 
\end{split}
\label{est-uh-open-002} 
\end{equation}
and then 
\begin{align*}
\frac{d}{dt} \| u _{h_1}  (t)  - u _{h_2}  (t)  \| ^2 &\leq 2 \| u _{h_1}  (t)  - u _{h_2}  (t) \|   \| u' _{h_1}  (t)  - u' _{h_2}  (t) \| \\
		& \leq 
		\frac{2}{\tau _0 } \LC 	\frac{1}{\lambda ^ {1/2 }  }  +   a \ell (T, M , f, u _ 0  , \tau _ 0  )  +  L + 1 \RC   \| u _{h_1}  (t)  - u _{h_2}  (t) \| ^2 \\
		& \hspace{1cm}  + \frac{1}{2\tau _0  }    \| h_1 (t) - h_2 (t) \| ^2  	.		 
\end{align*}
By the Gronwall's inequality, 
\begin{equation*}
\sup _{0\leq t \leq T }   \| u _{h_1}  (t)  - u _{h_2}  (t) \| ^2 
		\leq  \frac{e ^{C_6T} }{2\tau _0 }   \int_{0}^{T}  \| h_1 (t) - h_2 (t) \| ^2 dt 
\end{equation*}
holds, where 
\begin{equation*}
C_6 := \frac{2}{\tau _0 } \LC 	\frac{1}{\lambda ^ {1/2 }  }  +   a \ell (T, M , f, u _ 0  , \tau _ 0 )  +  L + 1 \RC . 
\end{equation*}
By using \eqref{est-uh-open-002}, we have 
\begin{align*}
  &  \int_{0}^{T} \| u' _{h_1}  (t)  - u' _{h_2}  (t)  \| ^2 dt \\
					 \leq  & ~ 
\frac{4}{\tau ^2  _ 0 } 
\LC 
\LC \frac{1}{\lambda   } +  a^2 \ell (T, M , f, u _ 0  , \tau _ 0 ) ^2 + L ^2 \RC \int_{0}^{T}   \| u _{h_1}  (t)  - u _{h_2}  (t) \| ^2 dt  \right . \\
& \hspace{3cm} 
\left . 
+ \int_{0}^{T}  \| h_1 (t) - h_2 (t) \| ^2 dt
\RC \\
 \leq  &~ 
\frac{4}{\tau ^2  _ 0 } 
\LC 
\LC \frac{1}{\lambda   } +  a^2 \ell (T, M , f, u _ 0  , \tau _ 0 ) ^2 + L ^2 \RC \frac{T e ^{C_6T} }{2\tau _0 }  +1  
\RC \int_{0}^{T}  \| h_1 (t) - h_2 (t) \| ^2 dt . 
\end{align*}
Therefore, if $\delta > 0 $ is sufficiently small so that 
\begin{equation*}
\delta ^2 < 
\frac{\tau ^2  _ 0 }{ 4} 
\LC 
\LC \frac{1}{\lambda   } +  a^2 \ell (T, M , f, u _ 0  , \tau _ 0 ) ^2 + L ^2 \RC \frac{T e ^{C_6T} }{2\tau _0 }  +1  
\RC ^{-1} , 
\end{equation*}
then $\mathscr{T} _{\delta } : L ^2 (0,T ; \R ^N ) \to L ^2 (0,T ; \R ^N ) $
is  a contraction mapping on $K ' _M$, whence follows Theorem \ref{Th3}.  
\end{proof}

\begin{proof}[Proof of Theorem \ref{Th2}]
To show the solvability of the original problem \eqref{Eq},  
we discuss the convergence of \eqref{EqLam_tau} as $\lambda \to 0 $. 
This argument is almost the same as that for the case of $\tau = 0 $ (recall Section 3).
Let $( u _{\lambda } , \mu _{\lambda })$ be a unique solution to \eqref{EqLam_tau}. 
Multiplying the first equation of \eqref{EqLam_tau} by $\mu _{\lambda }(t) $
and 
multiplying the second equation of \eqref{EqLam_tau} by $u ' _{\lambda } (t)$,
we have 
\begin{equation*}
\frac{\tau}{2}  \| u '  _{\lambda } (t)\| ^2 +
\frac{d}{dt} \varphi _\lambda  (u _{\lambda }(t))
+ \frac{a}{q} \frac{d}{dt} \| u _{\lambda }(t) \| ^q 
+  \frac{d}{dt} k (u _{\lambda } (t)) \leq  \frac{1}{2 \tau } \| f(t) \|  ^2  . 
\end{equation*}
Integrating this over $[0,T]$, we get 
\begin{align*}
&\frac{\tau}{2} \int_{0}^{T}  \| u '  _{\lambda } (t)\| ^2 dt 
+ \frac{a}{q} \| u _{\lambda }(T) \| ^q 
+  k (u _{\lambda } (T))  \\
&\hspace{2cm} \leq  \frac{1}{2 \tau } \int_{0}^{T}   \| f(t) \|  ^2 dt 
+
\varphi _{G,p  }  (u _0 )
+ \frac{a}{q} \| u _ 0  \| ^q 
+   k (u _0 )  . 
\end{align*}
From \eqref{est-of-k} and $q > 2 $, 
\begin{equation}
 \int_{0}^{T}  \| u '  _{\lambda } (t)\| ^2 dt \leq C_7 , 
\label{est-lam-to-0-001} 
\end{equation}
where $C_7 > 0 $ is a general constant independent of $\lambda $. 
This implies 
\begin{equation}
\sup _{0 \leq t \leq T } \| u   _{\lambda } (t)\|  \leq C_7 . 
\label{est-lam-to-0-002} 
\end{equation}
By the second equation of \eqref{EqLam_tau}, 
\begin{equation}
\int_{0}^{T}  \| \mu   _{\lambda } (t)\|  ^2 dt  \leq C_7 . 
\label{est-lam-to-0-003} 
\end{equation}
From the first equation of \eqref{EqLam_tau} and \eqref{bounded}, 
we also derive 
\begin{equation}
 \sup _{0\leq t \leq T  } \| \partial \varphi _{\lambda } ( u  _{\lambda }(t) ) \| 
+
\int_{0}^{T}  \| \partial \varphi _{\lambda } ( \mu  _{\lambda }(t) ) \| ^2 dt  
\leq 
C _7 . 
\label{est-lam-to-0-004}
\end{equation}

By \eqref{est-lam-to-0-001}--\eqref{est-lam-to-0-004},
there exist some subsequences $ \{ u _{\lambda _n  } \} _ {n \in \N } ,  \{ \mu _{\lambda _n  } \} _ {n \in \N } $
and some $u , \mu , \xi , \zeta : [0,T ] \to \R ^N $ such that 
\begin{align*}
& u _{\lambda _n } \to u  & \text{strongly in } C( [0,T] ; \R ^N ) ,\\
& u' _{\lambda _n } \rightharpoonup  u'  & \text{weakly in } L ^2
			( 0,T ; \R ^N ) , \\
& \mu _{\lambda _n } \rightharpoonup  \mu  & \text{weakly in } L ^2
			( 0,T ; \R ^N ) ,\\
& \partial \varphi _{\lambda _n } ( u  _{\lambda _n }(t) ) \rightharpoonup \xi
		& \ast \text{-weakly in } L ^{\infty }
			( 0,T ; \R ^N ) , \\
& \partial \varphi _{\lambda _n } ( \mu  _{\lambda _n }(t) ) \rightharpoonup  \zeta
		&  \text{weakly in } L ^2
			( 0,T ; \R ^N ) .
\end{align*}
By Lemma \ref{demiclosedness}, 
we have $\xi (t) \in \partial \varphi _{G,p } (u (t)) $ for a.e. $t\in [0,T]$.
Then $ u, \mu $ and $\zeta $ satisfy 
\begin{equation*}
\begin{cases}
~~ u' (t) + \zeta (t) = 0 ,\\
~~\mu (t) = \tau u' (t) + \xi (t) + a \| u(t) \| ^{q-2} u(t)
			+ \pi (u(t)) - f(t) .
\end{cases}
\end{equation*}
To show that $\zeta (t) \in \partial \varphi _{G,p } (\mu (t)) $ holds for a.e. $t\in [0,T]$, 
we multiply the second equation of \eqref{EqLam} with $\lambda = \lambda _ n $
 by $u ' _{\lambda _n }$ and take the limit $n \to \infty $.
 Then we have 
\begin{align*}
&\liminf _{n \to \infty } \int_{0}^{T} \mu _{\lambda _n } (t) \cdot u' _{\lambda _n} (t) dt\\
=&  ~ \liminf _{n \to \infty }
\int_{0}^{T} 
\LC 
\frac{d}{dt }\varphi _{\lambda _n } (u _{\lambda _n } (t)  )
+ \frac{a}{q}\frac{d}{dt } \| u _{\lambda _n } (t) \| ^q 
+ \frac{d}{dt } k (u _{\lambda _n} (t))
-u ' _{\lambda _n } (t) \cdot f(t)   
\RC dt \\
=& ~  \liminf _{n \to \infty }
\LC 
\varphi _{\lambda _n } (u _{\lambda _n } (T)  )
-\varphi _{\lambda _n } (u _0 )
+ \frac{a}{q}  \| u _{\lambda _n } (T) \| ^q 
- \frac{a}{q}  \| u _0 \| ^q \right . \\
&\left . \hspace{5cm} 
+  k (u _{\lambda _n} (T))
-  k (u _0)
-\int_{0}^{T} 
 u ' _{\lambda _n } (t) \cdot f(t)  dt  
\RC  \\
\geq 
& ~
 \liminf _{n \to \infty }\varphi _{\lambda _n } (u _{\lambda _n } (T)  )
-\varphi _{G,p } (u _0 )
+ \frac{a}{q}  \| u  (T) \| ^q 
- \frac{a}{q}  \| u _0 \| ^q  \\
&\hspace{5cm}
+  k (u  (T))
-  k (u _0)
-\int_{0}^{T} 
 u ' (t) \cdot f(t)  dt  . 
\end{align*}
Let  $J _{\lambda }$ be the resolvent of $\partial \varphi _{G,p }$. 
Then by 
\eqref{est-lam-to-0-004}, we have 
\begin{equation*}
\sup _{0 \leq t \leq T} \| J _{\lambda _n } u _{\lambda _n } (t) - u _{\lambda _n } (t) \|
\leq \lambda _n C _7 . 
\end{equation*}
Hence 
\begin{align*}
& J _{\lambda _n }u _{\lambda _n } \to u  & \text{strongly in } C( [0,T] ; \R ^N ) .
\end{align*}
By the continuity of $\varphi _{G,p }$
and 
$\varphi _{G,p } (J _{\lambda _n }x) \leq  \varphi _{\lambda _n } (x) $, we get 
\begin{align*}
 \liminf _{n \to \infty }\varphi _{\lambda _n } (u _{\lambda _n } (T)  )
& \geq 
 \liminf _{n \to \infty }\varphi _{G, p } ( J _{\lambda _n }u _{\lambda _n } (T)  ) \\
& \geq 
 \varphi _{G, p } ( u (T)  ) .
\end{align*}
Therefore 
\begin{align*}
&\liminf _{n \to \infty } \int_{0}^{T} \mu _{\lambda _n } (t) \cdot u' _{\lambda _n} (t) dt\\
\geq 
& ~
 \varphi _{G, p } ( u (T)  ) 
-\varphi _{G,p } (u _0 )
+ \frac{a}{q}  \| u  (T) \| ^q 
- \frac{a}{q}  \| u _0 \| ^q 
+  k (u  (T))
-  k (u _0)
-\int_{0}^{T} 
 u ' (t) \cdot f(t)  dt  \\
=
&
\int_{0}^{T}  
\LC
\frac{d}{dt}  \varphi _{G, p } ( u (t)  ) 
+ \frac{a}{q} \frac{d}{dt}   \| u  (t) \| ^q 
+  \frac{d}{dt}  k (u  (t))
-
 u ' (t) \cdot f(t) \RC dt  \\
=
&
\int_{0}^{T}  
u ' (t) 
\cdot 
\LC
\xi (t) 
+a \| u(t) \| ^{q-2} u(t) 
+\pi (u (t))
- f(t) \RC dt  
=
\int_{0}^{T}  
u ' (t) 
\cdot 
\mu (t) dt ,
\end{align*}
which yields 
\begin{align*}
&\limsup _{n \to \infty } \int_{0}^{T} J_{\lambda _n }\mu _{\lambda _n } (t) \cdot u '_{\lambda _n} (t) dt
\geq 
\liminf _{n \to \infty } \int_{0}^{T} J_{\lambda _n }\mu _{\lambda _n } (t) \cdot u' _{\lambda _n} (t) dt
\\
 = & ~
\liminf _{n \to \infty } 
\LC
\int_{0}^{T} \mu _{\lambda _n } (t) \cdot u' _{\lambda _n} (t) dt
+
\int_{0}^{T} ( J_{\lambda _n }\mu _{\lambda _n } (t)
- \mu _{\lambda _n } (t) )  \cdot u' _{\lambda _n} (t) dt
\RC
 \\
\geq & ~   
\liminf _{n \to \infty } 
\LC
\int_{0}^{T} \mu _{\lambda _n } (t) \cdot u ' _{\lambda _n} (t) dt
- \lambda _n 
\int_{0}^{T} \| \partial \varphi _{\lambda _n } (\mu _{\lambda _n })\|
 \|  u' _{\lambda _n} (t) \|  dt
\RC
 \\
= & ~
\int_{0}^{T}  
u ' (t) 
\cdot 
\mu (t) dt . 
\end{align*}
On the other hand, by \eqref{est-lam-to-0-004}, 
\begin{equation*}
\int_{0}^{T} \| J_{\lambda _n }\mu _{\lambda _n } (t) -
\mu _{\lambda _n } (t) \| ^2 dt 
= 
\lambda ^2_n \int_{0}^{T}  \| \partial \varphi _{\lambda _n } ( \mu _{\lambda _n } (t) )
 \| ^2 dt \to 0 
\end{equation*}
as $n \to \infty $. Hence 
\begin{align*}
& J _{\lambda _n } \mu  _{\lambda _n } \rightharpoonup  \mu   & \text{weakly in } L ^2 ( 0,T ; \R ^N ) .
\end{align*}
From the first equation of \eqref{EqLam}, we can derive 
\begin{align*}
\limsup _{n \to \infty } \int_{0}^{T} J_{\lambda _n }\mu _{\lambda _n } (t) \cdot 
\partial \varphi _{\lambda _n } ( \mu _{\lambda _n} (t) ) dt
& =
\limsup _{n \to \infty } \int_{0}^{T} J_{\lambda _n }\mu _{\lambda _n } (t) \cdot 
( - u' _{\lambda _n} (t) ) dt \\
& =
- \liminf _{n \to \infty } \int_{0}^{T} J_{\lambda _n }\mu _{\lambda _n } (t) \cdot 
 u' _{\lambda _n} (t)  dt \\
& \leq 
- \int_{0}^{T} \mu  (t) \cdot 
 u'  (t)  dt \\
& \leq 
 \int_{0}^{T} \mu  (t) \cdot 
 \zeta  (t)  dt . 
\end{align*}
Since 
$\partial \varphi _{\lambda _n } ( \mu _{\lambda _n} (t) ) 
\in \partial \varphi _{G,p  } (J _{\lambda _n } \mu _{\lambda _n} (t) )$
for a.e. $t \in [0 , T]$, 
Lemma \ref{wwlimsup} leads to 
\begin{equation*}
\zeta (t) \in \partial \varphi _{G,p } (\mu (t)) ~~~~~\text{a.e. } t\in [0,T] , 
\end{equation*}
whence follows Theorem \ref{Th2}. 
\end{proof}

\section{Convergence as $\tau \to 0 $}

In this section, we discuss the convergence of solutions of  
\eqref{Eq} with $\tau > 0 $ tending to that with $\tau  \to 0 $. 
Henceforth, we assume $f \equiv 0 $. 
\begin{Th}
\label{Th4}
Assume that  $f \equiv 0 $. 
Let a solution to \eqref{Eq} with $\tau \in ( 0,1 ) $ be $ ( u _ \tau , \mu _ \tau  ) $. 
Then there exist some subsequences $\{ \tau _n \} _{n\in \N }$, $ \{ u _{\tau _n }\} _{n\in \N }$, $ \{ \mu _{\tau _n }\} _{n\in \N } $
and $u , \mu : [0,T] \to \R ^N $ such that
 $\tau _n \to 0 $ as $n \to \infty $ and
 for any fixed $T> 0 $
\begin{align*}
& u _{\tau _n } \to u  & \text{strongly in } C( [0,T] ; \R ^N ) ,\\
& u' _{\tau _n } \rightharpoonup  u'  & \ast \text{-weakly in } L ^{\infty }
			( 0,T ; \R ^N ) ,\\
& \mu  _{\tau _n } \rightharpoonup  \mu   & \ast \text{-weakly in } L ^{\infty }
			( 0,T ; \R ^N ) ,
\end{align*}
if $p =1 $ and 
\begin{align*}
& u _{\tau _n } \to u  & \text{strongly in } C( [0,T] ; \R ^N ) ,\\
& u' _{\tau _n }\rightharpoonup  u'  & \text{weakly in } L ^{p' }
			( 0,T ; \R ^N ) ,\\
& \mu  _{\tau _n } \rightharpoonup  \mu   & \text{weakly in } L ^{ \bar{p}  }
			( 0,T ; \R ^N ) , 
\end{align*}
if $p > 1  $, where 
$ p' : = p / ( p-1 )$ is the H\"{o}lder conjugate of $p  > 1 $ and  
$\bar{p} := \max \{ p , p' \} $. 
Furthermore, the pair of limits $ ( u , \mu ) $ is a solution to \eqref{Eq} with $\tau = 0 $. 
\end{Th}

\begin{proof}
Multiplying  the first equation of \eqref{Eq} by $ \mu  _{\tau }$
and 
multiplying  the second  equation of \eqref{Eq} by $ u '  _{\tau }$,
we get from \eqref{timesx}
\begin{equation}
p \varphi _{G,p } (\mu _ {\tau }  (t) )  + \tau \| u '  _ {\tau } (t)\| ^2 
+ \frac{d}{dt}  \varphi _{G,p } (u  _ {\tau }  (t) )
+ \frac{a}{q} \frac{d}{dt}  \| u  _ {\tau }  (t) \| ^q 
+  \frac{d}{dt}  k (  u  _ {\tau }  (t) ) \leq  0 ,
\label{est-tau-to-000} 
\end{equation}
and then 
\begin{equation*}
 \frac{a}{q}   \| u  _ {\tau }  (t) \| ^q 
+    k (  u  _ {\tau }  (t) ) \leq 
  \varphi _{G,p } (u  _ 0 )
+ \frac{a}{q}   \| u  _ 0  \| ^q 
+    k (  u  _0  ) . 
\end{equation*}
By \eqref{est-of-k} and $q >2 $, we have 
\begin{equation}
\sup _{ 0\leq t \leq T }   \| u  _ {\tau }  (t) \|  \leq C_8 , 
\label{est-tau-to-001} 
\end{equation}
where $c_8 , C_8 >0  $ are general constants independent of $\tau $. 
From \eqref{est-tau-to-000}, we also derive 
\begin{equation}
\tau  \int_{0}^{T}  \| u '  _ {\tau } (t)\| ^2 dt \leq C _ 8 
\label{est-tau-to-002} 
\end{equation}
and 
\begin{equation}
\int_{0}^{T}  \varphi _{G,p } (\mu _ {\tau }  (t) ) dt \leq C_ 8 . 
\label{est-tau-to-003} 
\end{equation}

Here let  $ \bm{1}  = (1, 1, \ldots, 1 )$ and 
\begin{equation*}
\bar{u} _{\tau}  (t) := \LC \frac{1}{N}  \sum_{i=1}^{N } u_{\tau i} (t) \RC \bm{1} . 
\end{equation*}
From \eqref{times1},  
multiplying the first equation of  \eqref{Eq} by $\bm{1}$, 
we get 
\begin{equation*}
\bar{u} '_{\tau } (t) = 0 .  
\end{equation*}
Multiplying the second equation of  \eqref{Eq} by $\bm{1}$, 
we have 
\begin{equation*}
\bar{\mu } _{\tau } (t) =  a \| u _{\tau } (t) \| ^{q-2} \bar{u} _{\tau } (t)
			+ \overline{\pi (u _{\tau }(t)) }  .  
\end{equation*}
Hence by \eqref{est-tau-to-001}, 
\begin{equation}
\sup _{0\leq t\leq T } \| \bar{\mu } _{\tau } (t) \|  \leq  C_8 .  
\label{est-tau-to-004} 
\end{equation}
By \eqref{Poincare}, we have 
\begin{align*}
\varphi _{G,p  } (\mu _{\tau } (t) ) 
& \geq c_8 \|  \mu _\tau (t) - \bar{ \mu} _\tau (t) \| ^ p  \\
& \geq c_8 \|  \mu _\tau (t) \| ^p - C _8 \| \bar{ \mu} _\tau (t) \| ^ p   . 
\end{align*}
From \eqref{est-tau-to-000}, 
\begin{align*}
c_8 \|  \mu _\tau (t) \| ^p
+ \frac{d}{dt}  \varphi _{G,p } (u  _ {\tau }  (t) )
+ \frac{a}{q} \frac{d}{dt}  \| u  _ {\tau }  (t) \| ^q 
+  \frac{d}{dt}  k (  u  _ {\tau }  (t) ) \leq 
C _8 \| \bar{ \mu} _\tau (t) \| ^ p  , 
\end{align*}
which implies 
\begin{equation}
\int_{0}^{T}  \|  \mu _\tau (t) \| ^p dt \leq C_8 . 
\label{est-tau-to-005} 
\end{equation}

Let $\xi  _ {\tau }$  and $ \zeta _{\tau }$ be the sections of $\partial  \varphi _{G,p } (u _{\tau } )$
and
$\partial  \varphi _{G,p } (\mu _{\tau } )$ satisfying the 1st equation of \eqref{Eq}. 
Namely, 
$\xi  _ {\tau }$  and $ \zeta _{\tau }$  fulfill 
\begin{align*}
& \xi _{\tau }  (t) \in \partial  \varphi _{G,p } ( u _{\tau } (t) ) , ~~~~
\zeta _{\tau } (t) \in \partial  \varphi _{G,p } (\mu _{\tau } (t)) , \\
&  u' _{\tau } (t) + \zeta _{\tau }(t) =  0 , ~~~~~
 \mu_{\tau }  (t) = \tau u' _{\tau }(t) + \xi _{\tau } (t) + a \| u _{\tau } (t) \| ^{q-2} u_{\tau }(t)
			+ \pi (u_{\tau } (t)) 
\end{align*}
for a.e. $t \in [0 ,T ]$. 
From \eqref{bounded}, 
we have 
\begin{equation*}
\| u ' _{\tau }(t) \| = \| \zeta  _{\tau } (t) \| \leq C_8 \| \mu _{\tau } (t) \| ^{p-1} .  
\end{equation*}
Hence when $p =1$, it follows that 
\begin{equation}
\sup _{0 \leq t \leq T } \| u ' _{\tau }(t) \| \leq C_8  
\label{est-tau-to-006} 
\end{equation}
and 
from the second equation of \eqref{Eq} and \eqref{est-tau-to-001}, 
\begin{equation}
\sup _{0 \leq t \leq T } \| \mu  _{\tau }(t) \| \leq C_8 .
\label{est-tau-to-006-1} 
\end{equation}
 When $p >1 $, 
\begin{equation}
\int_{0}^{T} \| u  '  _{\tau }(t) \| ^{ p' } dt \leq  C_8 \int_{0}^{T}  \| \mu _{\tau } (t) \| ^{p} dt \leq C _8
\label{est-tau-to-007} 
\end{equation}
by \eqref{est-tau-to-002}, where $ p' : = p / ( p-1 )$ is the H\"{o}lder conjugate of $p > 1 $. 
From the second equation of \eqref{Eq}, we also get 
\begin{equation}
  \int_{0}^{T} \| \mu   _{\tau }(t) \| ^{ p' } dt \leq  C _8 .
\label{est-tau-to-007-1} 
\end{equation}

According to \eqref{est-tau-to-001} and \eqref{est-tau-to-006}, 
there exist some subsequence $\{ \tau  _n \} _ {n \in \N }$ of $\{ \tau  \} _{\tau > 0 }$
and $u : [0,T] \to \R ^N $ such that $\tau _n \to 0 $ as $n \to \infty $ and
\begin{align*}
& u _{\tau _n } \to u  & \text{strongly in } C( [0,T] ; \R ^N ) ,\\
& u' _{\tau _n } \rightharpoonup  u'  & \ast \text{-weakly in } L ^{\infty }
			( 0,T ; \R ^N ) ,
\end{align*}
as $\tau _n \to 0 $  when $ p =1 $. 
On the other hand, from 
\eqref{est-tau-to-001} and \eqref{est-tau-to-007}, 
we can derive 
\begin{align*}
& u _{\tau _n }  \to u  & \text{strongly in } C( [0,T] ; \R ^N ) ,\\
& u' _{\tau _n } \rightharpoonup  u'  & \text{weakly in } L ^{p' }
			( 0,T ; \R ^N ) ,
\end{align*}
as $\tau _n \to 0 $  when $p > 1 $.  
Moreover, by using \eqref{est-tau-to-005}, \eqref{est-tau-to-006-1}, and \eqref{est-tau-to-007-1}, 
there exists some $\mu : [0 , T ] \to \R ^N $ such that
\begin{align*}
& \mu  _{\tau _n } \rightharpoonup  \mu   & \ast \text{-weakly in } L ^{\infty }
			( 0,T ; \R ^N ) 
\end{align*}
if $p=1 $ and 
\begin{align*}
& \mu  _{\tau _n } \rightharpoonup  \mu   & \text{weakly in } L ^{ \bar{p}  }
			( 0,T ; \R ^N ) 
\end{align*}
if $p> 1$, where $\bar{p} := \max \{ p , p' \} $.

Let $ \xi _{\tau _n }  $ and $ \zeta _{\tau _n } $ 
be sections of $ \partial \varphi _{G,p } (u _{\tau _n } )$ and  $ \partial \varphi _{G,p } (\mu  _{\tau _n } )$
satisfying the system \eqref{Eq} with $\tau = \tau _n $, respectively.
Namely, $ \xi _{\tau _n } ,  \zeta _{\tau _n } : [0,T ] \to \R ^N $
satisfy 
$ \xi _{\tau _n } (t) \in \partial \varphi _{G,p } (u _{\tau _n } (t) )$,  
$\zeta _{\tau _n} (t) \in  \partial \varphi _{G,p } (\mu  _{\tau _n }  (t ))$, and 
\begin{equation*}
\begin{cases}
~~ u' _{\tau _n } (t) + \zeta   _{\tau _n } (t)  = 0  , \\
~~\mu  _{\tau _n }(t) = \tau _n u' _{\tau _n }(t) + \xi  _{\tau _n } (t) + a \| u _{\tau _n } (t) \| ^{q-2} u_{\tau _n }(t)
			+ \pi (u _{\tau _n } (t)) ,
\end{cases}
\end{equation*}
for a.e. $t \in [0,T ]$. 
By \eqref{bounded} and \eqref{est-tau-to-001}, we obtain 
\begin{equation*}
\sup _{0\leq t \leq T } \| \xi _{\tau _n} (t) \| \leq C_8  
\end{equation*}
and there exists some $\xi :[0,T] \to \R ^N $ such that 
\begin{align*}
& \xi   _{\tau _n } \rightharpoonup  \xi  & \ast \text{-weakly in } L ^{\infty }
			( 0,T ; \R ^N ) 
\end{align*}
(we omit relabeling). 
From Lemma \ref{demiclosedness}, we derive 
$ \xi (t) \in \partial \varphi _{G,p } (u (t))$ for a.e. $t\in [0,T]$. 
By \eqref{bounded} and \eqref{est-tau-to-006-1}, 
\begin{equation*}
\sup _{0\leq t\leq T } \| \zeta _{\tau _n } (t) \| \leq C_ 8 
\end{equation*}
and 
\begin{align*}
& \zeta    _{\tau _n } \rightharpoonup  \zeta   & \ast \text{-weakly in } L ^{\infty }
			( 0,T ; \R ^N ) 
\end{align*}
when $p=1$. 
On the other hand, when $p>1$, 
\begin{equation*}
\int_{0}^{T} \|  \zeta _{\tau _n} (t) \| ^{p'} dt 
\leq 
C_8  \int_{0}^{T} \|  \mu _{\tau _n} (t) \| ^{p} dt 
\leq C_8 
\end{equation*}
holds by \eqref{est-tau-to-005}, which implies
\begin{align*}
& \zeta    _{\tau _n } \rightharpoonup  \zeta   &  \text{weakly in } L ^{p' }
			( 0,T ; \R ^N ) .
\end{align*}
By \eqref{est-tau-to-002}, we get 
\begin{align*}
& \tau _{n } u '    _{\tau _n } \to 0   &  \text{strongly in } L ^{2 }
			( 0,T ; \R ^N ) 
\end{align*}
as $\tau _n \to 0 $.
Hence the limits $ u , \mu , \xi , \zeta $ satisfy 
\begin{equation*}
\begin{cases}
~~ u'  (t) + \zeta   (t)  = 0  , \\
~~\mu  (t) =  \xi   (t) + a \| u  (t) \| ^{q-2} u  (t)
			+ \pi (u (t)) .
\end{cases}
\end{equation*}

Finally,  we have to prove that $\zeta (t) \in \partial \varphi _{G,p } (\mu (t)) $ holds for a.e. $t\in [0,T]$.
Multiplying the second equation of \eqref{EqLam} with $\tau  = \tau _ n $, 
 we have by $\tau _n \| u' _{\tau _n} (t) \| ^2 \geq 0  $
\begin{align*}
&\liminf _{n \to \infty } \int_{0}^{T} \mu _{\tau _n } (t) \cdot u' _{\tau _n} (t) dt \\
= & ~   \liminf _{n \to \infty }
\int_{0}^{T} 
\LC 
\tau _n \| u' _{\tau _n} (t) \| ^2 
+
\frac{d}{dt }\varphi _{G,p  } (u _{\tau  _n } (t)  )
+ \frac{1}{q}\frac{d}{dt } \| u _{\tau  _n } (t) \| ^q 
+ \frac{d}{dt } k (u _{\tau  _n} (t)) 
\RC dt \\
 \geq &~ 
\varphi _{ G,p  } (u  (T)  )
-\varphi _{G,p } (u _0 )
+ \frac{a}{q}  \| u  (T) \| ^q 
- \frac{a}{q}  \| u _0 \| ^q 
+  k (u  (T))
-  k (u _0) \\
 =&  ~ 
\int_{0}^{T}  
\LC 
\xi (t) 
+ 
a \|  u(t) \| ^{q-2} u (t) + \pi (u (t ) )   
\RC
\cdot u' (t) dt \\
 = &~ 
\int_{0}^{T}  
\mu (t) \cdot u' (t) dt . 
\end{align*}
Therefore, 
\begin{equation*}
\limsup _{n\to \infty } \int_{0}^{T} \mu _{\tau _n } (t) \cdot \zeta _{\tau _n} (t) dt
\leq 
\limsup _{n\to \infty } \int_{0}^{T} \mu  (t) \cdot \zeta (t) dt
\end{equation*}
holds, whence follows that $ \zeta (t) \in \partial \varphi _{G,p } ( \mu (t) )$ for a.e. $t\in [0,T]$.
\end{proof}


\pagebreak 

%
\address{
Takeshi Fukao\\
Faculty of Advanced Science and \\
Technology, Ryukoku University, \\
1-5 Yokotani, Seta Oecho, Otsu, \\
520-2194, JAPAN.%
}
{fukao@math.ryukoku.ac.jp}
%
\address{
Masahiro Ikeda\\
Graduate School of Information \\
Science and Technology, \\
The University of Osaka, \\
1-5 Yamadaoka Suita-shi, Osaka \\
565-0871, JAPAN/\\
Center for Advanced Intelligence \\
Project, RIKEN,
1-4-1 Nihonbashi, \\
Chuo-ku, 103-0027, JAPAN.
}
{ikeda@ist.osaka-u.ac.jp/\\
masahiro.ikeda@a.riken.jp}
%
%
\address{
Shun Uchida\\
Department of Integrated Science and Technology, \\
Faculty of Science and Technology, \\ 
Oita University,\\
700 Dannoharu, Oita City, Oita Pref., \\
 870-1192, JAPAN.
}
{shunuchida@oita-u.ac.jp}
\end{document}